\documentclass[11pt]{amsart}
\usepackage{pgf,tikz,pgfplots,color}
\pgfplotsset{compat=1.15}
\usepackage{mathrsfs} %mathscr
\usetikzlibrary{arrows}
\usepackage{fancyhdr}
\usepackage{lastpage} 

\usepackage[margin=1in]{geometry} 
\usepackage{geometry}
\usepackage[toc,page]{appendix}
\usepackage[utf8]{inputenc}
\usepackage{hyperref}
\hypersetup{
	colorlinks=true,
	linkcolor=blue,
	citecolor=brown,
	filecolor=magenta,
	urlcolor=blue,
}
\usepackage[notcite,notref]{} 
\usepackage{amsmath}
\usepackage{amsfonts}
\usepackage{amssymb}
\usepackage{amsthm}
\usepackage{setspace}
\usepackage{inputenc}
\usepackage[english]{babel} 
\usepackage{comment}
\usepackage[shortlabels]{enumitem}
\numberwithin{equation}{section}

\title[]{A Solution Operator for the $\overline\partial$ Equation in Sobolev Spaces of Negative Index}          

\author[]{Ziming Shi}
\address{Ziming Shi (Corresponding author), Department of Mathematics,
University of California - Irvine, Irvine, CA, 92697}  
\email{zimings3@uci.edu}
\author[]{Liding Yao} 
\address{Liding Yao, Department of Mathematics,
	The Ohio State University, Columbus, OH 43210} 
\email{yao.1015@osu.edu}
\newcommand{\dist}{\operatorname{dist}}

\newcommand{\supp}{\operatorname{supp}}
\newcommand{\loc}{\mathrm{loc}}

\newtheorem{thm}{Theorem}[section]
\newtheorem{cor}[thm]{Corollary} 
\newtheorem{prop}[thm]{Proposition}
\newtheorem{lemma}[thm]{Lemma}
\theoremstyle{definition}
\newtheorem{defn}[thm]{Definition}
\newtheorem{exmp}[thm]{Example}
\newtheorem{ques}[thm]{Question}
\newtheorem{note}[thm]{Notation}

\theoremstyle{remark}
\newtheorem{rem}[thm]{Remark}
\newtheorem*{clm}{Claim}
\newtheorem*{ack}{Acknowledgment}
\renewcommand{\th}[1]{\begin{thm}\label{#1}}
	\renewcommand{\eth}{\end{thm}}
\newcommand{\co}[1]{\begin{cor}\label{#1}}
	\newcommand{\eco}{\end{cor}}
\newcommand{\pr}[1]{\begin{prop}\label{#1}}
	\newcommand{\epr}{\end{prop}}
\newcommand{\df}[1]{\begin{defn}\label{#1}}
	\newcommand{\edf}{\end{defn}}
\newcommand{\ex}[1]{\begin{exmp}\label{#1}} 
	\newcommand{\eex}{\end{exmp}}
\newcommand{\qu}[1]{\begin{ques}\label{#1}}
	\newcommand{\equ}{\end{ques}}  
\newcommand{\mk}{\begin{rem}}
	\newcommand{\emk}{\end{rem}}
\newcommand{\cl}{\begin{clm}}
	\newcommand{\ecl}{\end{clm}} 
\newcommand{\ac}{\begin{ack}}
	\newcommand{\eac}{\end{ack}} 

\newcommand{\ga}{\begin{gather}}
\newcommand{\ega}{\end{gather}}
\newcommand{\gan}{\begin{gather*}}
\newcommand{\egan}{\end{gather*}}
\newcommand{\al}{\begin{gngn}}
	\newcommand{\eal}{\end{align}}
\newcommand{\aln}{\begin{align*}}
\newcommand{\ealn}{\end{align*}}
\newcommand{\eq}[1]{\begin{equation}\label{#1}}
\newcommand{\eeq}{\end{equation}}

\newcommand{\pa}{\partial{}}
\newcommand{\na}{\nabla}
\newcommand{\db}{\dbar}

\newcommand{\we}{\wedge}

\newcommand{\ra}{\longrightarrow}

\newcommand{\sm}{\setminus}
\newcommand{\seq}{\subseteq}

\newcommand{\DD}[2]{\frac{\partial #1}{\partial #2}}

\newcommand{\B}{\mathbb{B}} 
\newcommand{\Z}{\mathbb{Z}}
\newcommand{\R}{\mathbb{R}} 
\newcommand{\C}{\mathbb{C}}

\newcommand{\N}{\mathbb{N}}

\newcommand{\V}{\mathcal{V}}
\newcommand{\U}{\mathcal{U}} 

\newcommand{\mc}{\mathcal}
\newcommand{\mf}{\mathfrak}

\newcommand{\tit}{\textit}

\newcommand{\ov}{\overline}
\newcommand{\ti}{\tilde}
\newcommand{\wti}{\widetilde}
\newcommand{\hht}{\widehat}

\newcommand{\dbar}{\overline\partial}

\newcommand{\all}{\alpha}

\newcommand{\del}{\delta}
\newcommand{\Del}{\Delta}
\newcommand{\var}{\varphi}

\newcommand{\ve}{\varepsilon}
\newcommand{\om}{\omega}
\newcommand{\Om}{\Omega}
\newcommand{\thh}{\theta}
 
\newcommand{\La}{\Lambda}
\newcommand{\la}{\lambda}

\newcommand{\gm}{\gamma}

\newcommand{\si}{\sigma}

\newcommand{\yh}{\frac{1}{2}}

\newcommand{\re}[1]{(\ref{#1})}

\newcommand{\rl}[1]{Lemma~\ref{#1}}
\newcommand{\rc}[1]{Corollary~\ref{#1}}
\newcommand{\rp}[1]{Proposition~\ref{#1}}
\newcommand{\rt}[1]{Theorem~\ref{#1}}
\newcommand{\rd}[1]{Definition~\ref{#1}}

\newcommand{\nn}{\nonumber}
\newcommand{\nid}{\noindent}

\newcounter{pp}
\newcommand{\bpp}{\begin{list}{$\hspace{-1em}\alph{pp})$}{\usecounter{pp}}}
	\newcommand{\epp}{\end{list}}

\newcounter{ppp}
\newcommand{\bppp}{\begin{list}{$\hspace{-1em}(\roman{ppp})$}{\usecounter{ppp}}}
	\newcommand{\eppp}{\end{list}}

\newcommand{\Bs}{\mathscr{B}}

\newcommand{\Ec}{\mathcal{E}}
\newcommand{\Fc}{\mathcal{F}}
\newcommand{\Fs}{\mathscr{F}}

\newcommand{\Gf}{\mathfrak{G}} 

\newcommand{\Hc}{\mathcal{H}}

\newcommand{\Kb}{\mathbb{K}}
\newcommand{\Kc}{\mathcal{K}}

\newcommand{\Nb}{\mathbb{N}}
\newcommand{\Nc}{\mathcal{N}}
\newcommand{\Pc}{\mathcal{P}}
\newcommand{\Rf}{\mathfrak{R}}
\newcommand{\Sc}{\mathcal{S}}
\newcommand{\Ss}{\mathscr{S}}
\newcommand{\Uc}{\mathcal{U}}
\newcommand{\Vc}{\mathcal{V}}

\newcommand{\eps}{\varepsilon}
\begin{document}
	
	\begin{abstract} 
  Let $\Omega$ be a strictly pseudoconvex domain in $\mathbb{C}^n$ with $C^{k+2}$ boundary, $k \geq 1$. We construct a $\overline\partial$ solution operator (depending on $k$) that gains $\frac12$  derivative in the Sobolev space $H^{s,p} (\Omega)$ for any $1<p<\infty$ and  $s>\frac{1}{p} -k$. If the domain is $C^{\infty}$, then there exists a $\overline\partial$ solution operator that gains $\frac12$ derivative in $H^{s,p}(\Omega)$ for all $s \in \mathbb{R}$. We obtain our solution operators via the method of homotopy formula. A novel technique is the construction of ``anti-derivative operators'' for distributions defined on bounded Lipschitz domains.  
		
	\end{abstract} \maketitle 
	\tableofcontents
	
	\section{Introduction} 
	
The purpose of the paper is to prove the following results:
	\begin{thm} \label{Thm::Ck_bdy_intro}   
	Let $\Om$ be a bounded strictly pseudoconvex domain in $\C^n$ with $C^{k+2}$ boundary, where $n \geq 2$ and $k$ is a positive integer. For $1\le q\le n$, there is a linear operator $\Hc_q$ (depending on $k$) such that for all $1<p<\infty$ and $s>\frac1p-k$,
	\begin{enumerate}[(i)]
	    \item $\Hc_q:H^{s, p}_{(0,q)}(\Om)\to H^{s+ \yh, p}_{(0,q-1)}(\Om)$ is bounded.
	    \item $u=\Hc_q\var$ is a solution to the equation $\dbar u=\var$ for any $\var \in H^{s,p}_{(0,q)}(\Om)$ which is $\dbar$-closed.
	\end{enumerate}
	Here $H^{s,p}_{(0,q)}$ denotes the space of $(0,q)$ forms with $H^{s,p}(\Om)$ coefficients, and $H^{s,p}$ is the fractional Sobolev space (see \rd{Def:: Sobolev}).
	\end{thm}	
	
We also have the following result when the boundary is smooth. 
\begin{thm} \label{Thm::Cinfty_bdy_intro} 
	Let $\Om$ be a bounded strictly pseudoconvex domain in $\C^n$ with $C^\infty$ boundary, where $n \geq 2$. For $1\le q\le n$ there is a linear operator $\Hc_q$ such that for all $1<p<\infty$ and $s \in \R$,
	\begin{enumerate}[(i)]
	    \item $\Hc_q:H^{s, p}_{(0,q)}(\Om)\to H^{s+ \yh, p}_{(0,q-1)}(\Om)$ is bounded.
	    \item $u=\Hc_q\var$ is a solution to the equation $\dbar u=\var$ for any $\var \in H^{s,p}_{(0,q)}(\Om)$ which is $\dbar$-closed.
	\end{enumerate}
\end{thm}

The main novelty of our results is the existence of solutions when the given forms are in $H^{s,p}$ with negative (smoothness) index $s$, which can be viewed as the dual of certain positive indexed function space in the sense of distributions. 
This is new in $\dbar$ theory. In the $L^2$  and $\dbar$-Neumann theory for the $\dbar$ equation, the Hilbert space $L^2$, which is the $H^{0,2}$ space, plays a fundamental role in the methods of approach as well as the formulation of regularity results; see for example Kohn \cite{Koh63}, H\"ormander \cite{Hor65} and also the monographs \cite{Str_10} and \cite{C-S01}.  
	
	%In a vague interpretation, our regularity for all negative index relies on the fact that the $\dbar$ equation is not a boundary value problem. This is quite different from the PDE equations where boundary values are quite often imposed (see for instance for such restriction).

We now describe our approach when the given form $\var$ has negative index. We shall call the operator $\Hc_q$ in \rt{Thm::Ck_bdy_intro} a homotopy operator, since it satisfies the following $\db$ homotopy formula
\[
  \var = \dbar \Hc_q \var + \Hc_{q+1} \dbar \var. 
\]
In particular when $\var$ is a $\db$-closed form in $\Om$, $\Hc_q \var$ is a solution to the $\dbar$ equation. To construct our solution operator, we start with a $\dbar$ homotopy operator of the form 
 \begin{equation} \label{Hvar_intro}  
  \Hc \var (z) = \int_\Uc K_0(z,\zeta) \we \Ec \var(\zeta) + \int_{\Uc \sm \Om} K_1(z,\zeta) \we [\dbar, \Ec] \var, \quad [\dbar, \Ec] \var := \dbar \Ec \var - \Ec \dbar \var,    
 \end{equation} 
 where $\Uc$ is a neighborhood of $\ov \Om$ and 
$\Ec$ denotes some extension operator that extends a function or distribution class defined on $\Om$ to the same class in the whole space $\C^n$. A homotopy formula of this form was first constructed in \cite{Gong19} to prove the estimates for H\"older-Zygmund space on strictly pseudoconvex domains with $C^2$ boundary, and a similar modified form was introduced in \cite{S-Y21-1} to prove the estimates in Sobolev space $H^{s,p}$ of positive index $s$.

The kernel $K_0$ in \re{Hvar_intro} is a linear combination of the first derivative of the Newtonian potential, and consequently one can use standard theory of singular integrals to show that the first integral in \re{Hvar_intro} is in $H^{s+1,p}(\Om)$ for any $\var \in H^{s,p}(\Om)$ with $s \in \R$. 
For the second integral, we note that the kernel and its derivatives in $z$ will blow up as $\zeta \in \Uc \sm \Om$ approaches to the boundary $b \Om$. 
The important fact here is that the commutator $[\dbar, \Ec] \var$ vanishes identically inside $\Om$, and thus have a certain vanishing order as $\zeta \to b\Om$. Accordingly, one would expect to have some cancellations between the blow-up of the kernel and the vanishing of the commutator near $b\Om$, so that the integrals can be estimated.  
This is indeed the key idea used earlier in \cite{Gong19} and \cite{S-Y21-1}. However, in the present case, the commutator $[\dbar, \Ec]\var$ belongs to a Sobolev space with negative index, thus it exists only in the sense of distribution. In other words, the regularity is too low to even speak of the vanishing order. To remedy this, our idea is to express $[\dbar, \Ec]\var$ as the derivatives of more regular functions 
that lie in some positive index space. For this to work, we need to make sure these ``more regular" functions still have the important property that it vanishes identically in $\Om$. More specifically we construct a class of ``anti-derivative" operators on bounded Lipschitz domains.   
\begin{prop}[Anti-derivative operator with support condition] \label{Prop::anti-der_intro}
		Let $\Om$ be a bounded Lipschitz domain, and let $\Ss'(\R^N) $ denote the space of tempered distributions on $\R^N$. For any $m \in \Z^+$, there exist linear operators $\Sc_{\Om}^{m, \gm}: \Ss' (\R^N) \to \Ss'(\R^N)$, with $ |\gm| \leq m$ such that  
		\begin{enumerate}[(i)] 
			\item \label{Item::anti-der_intro::1}
			$\Sc^{m, \gm}_{\Om}: H^{s,p} (\R^N)  \to H^{s+m,p} (\R^N)$ for all $1 < p < \infty$ and $s \in \R$; 
			\item \label{Item::anti-der_intro::2}
			$
			f =  \sum_{|\gm| \leq m } D^{\gm} \Sc^{m,\gm}_{\Om} f,  \quad f \in \Ss'(\R^N);     
			$
			\item \label{Item::anti-der_intro::3}
			If $f \in \Ss'(\R^N)$ satisfies $f|_{\Om} \equiv 0$, then $(\Sc_{\Om}^{m, \gm} f)|_{\Om} \equiv 0$ for $|\gm| \leq m$.   
		\end{enumerate}
	\end{prop}
In other words, the operator $S^{m,\gm}_\Om$ gains $m$ derivatives. 
 We can now apply the proposition with $f$ being the commutator $[\dbar,\Ec] \var$, and write the second integral in \re{Hvar_intro} as 
\[
 \int_{\Uc \sm \Om} K_1(z,\zeta)\wedge\sum_{|\gm| \leq m }   D^{\gm} \Sc^{m,\gm}_{\Om} [\dbar, \Ec] \var (\zeta) \, dV(\zeta),  
\]
where $\Sc^{m,\gm}_{\Om} [\dbar, \Ec] \var$ is identically $0$ inside $\Om$ and is in $H^{s,p} (\C^n)$ for $s>0$ (by choosing $m$ sufficiently large).  
We still need to transport the extra derivatives $D^\gm$ to the kernel $K_1$, a procedure we call ``reverse integration by parts", since it moves derivatives to the kernel rather than away from it as is usually done with integration by parts. Finally the blow up of $D^\gm K_1$ cancels nicely with the vanishing order of $\Sc^{m,\gm}_{\Om} [\dbar, \Ec] \var $ near $b\Om$, and thus we can show that the integral is in $H^{s+\yh,p}(\Om)$ for any $\var \in H^{s,p}(\Om)$ with $s> \frac{1}{p}-k$.  

To control the blow-up order of derivatives at the singularities of the kernel $K_1$, we use the regularized Henkin-Ramirez functions introduced in (\cite[Section 5]{Gong19}) when the boundary is $C^{k+2}$, $k \geq 1$. For the case of $C^\infty$ boundary, the classical Henkin-Ramirez function is sufficient.  

\rp{Prop::anti-der_intro} is interesting in its own right and has other applications as well.  In a recent preprint \cite{S-Y21-2}, the authors generalize \rp{Prop::anti-der_intro} to Triebel-Lizorkin and Besov spaces. The results are then used to derive several important properties of the function spaces and the universal extension operator of Rychkov (see \cite{Ry99}). For the proof of \rp{Prop::anti-der_intro}, we remark that
  Condition \ref{Item::anti-der_intro::3} is the most non-trivial statement. To satisfy solely \ref{Item::anti-der_intro::1} and \ref{Item::anti-der_intro::2}, one can simply take $f=(I-\Delta)^m((I-\Delta)^{-m}f)=\sum_{|\gamma|\le 2m}c_\gamma D^\gamma ((I-\Delta)^{-m}f)$ on $\R^N$, where $(I-\Delta)^{-m}$ is the Bessel potential operator (see \eqref{BesselPotent}). Our construction of $\Sc^{m,\gm}_{\Om}$ is inspired by Rychkov's construction of the universal extension operator \cite{Ry99}. 
  The idea is that to break up $\hht f$ using a special Littlewood-Paley family $\{\la_j\}$, and then reduce it to a division problem in the frequency space. 
  %In particular, we need $\la_j$ to satisfy a) $\la_j$ are supported in some cone; and b) $\hht{\la_j}$ vanish to infinite order at the origin. Property (a) ensures that the anti-derivative operator is well-defined (locally) near Lipschitz boundary. Property (b) is used to smooth out the singularities resulted from the division and also prove the boundedness of the anti-derivative operator in Sobolev space. 

We now mention some brief history of the $1/2$ estimate for $\dbar$ equation in space of positive index, for which the theory has been well-studied and is quite satisfactory now. For $\Om$ with $C^{\infty}$ boundary, Kohn \cite{Koh63} showed that the $L^2$ canonical solution $\dbar^{\ast} \Nc \var$ is in $H^{s+\yh,2}(\Om)$ if $\var \in H^{s,2}(\Om)$ for any $s \geq 0$. Here $\Nc$ is the $\dbar$-Neumann operator. Later, Greiner-Stein \cite{G-S77} proved the $1/2$ estimate for $\dbar^{\ast} \Nc \var$ when $\var$ is a $(0,1)$ form of class $H^{k,p}(\Om)$, where $k$ is a non-negative integer and $1<p<\infty$. Their results were subsequently extended by Chang \cite{Cha89} to any $(0,q)$ form $\var$. 
	Also in \cite{G-S77}, it was proved that $\dbar^{\ast}\Nc \var \in \La^{r+\yh}(\ov \Om)$ for any $\var \in \La^r(\ov \Om)$, $r > 0$, where $\La^r$ is the H\"older-Zygmund space. 
	In both \cite{G-S77} and \cite{Cha89}, the boundary is required to be $C^{\infty}$.
	
In parallel with the $\dbar$-Neumann method, another theory has been developed to solve the $\db$ equation using integral kernels, in a way that
generalize the one dimensional Cauchy integral to higher dimension. Compared to the PDE approach used by the school of Kohn and Stein, the kernel method is more geometric in nature, and often times simpler to use. The solutions are not canonical in any sense, however they come in with a great variety of estimates, including some like the sup norm estimate which are not accessible by other methods. The theory began with the pioneering work of Grauert-Lieb \cite{G-L_70} and Kerzman \cite{Ker_71}. Soon afterwards, Henkin-Ramanov \cite{H-R71} constructed a solution operator on domains with $C^2$ boundary with the sharp $1/2$ gain in smoothness in H\"older space. The operator of Henkin-Romanov works for a $(0,1)$ form $\var \in C^0(\ov{\Om})$.  
Later on, Siu \cite{Siu74} (for $(0,1)$ form $\var$) and Lieb-Range \cite{L-R80} (for $(0,q)$ forms $\var$) constructed solutions in $C^{k+\yh}(\overline\Om)$ when $\var$ in $C^k(\overline\Om)$, under the assumption that boundary is $C^{k+2}$ for positive integer $k$. For a detailed history of the method of integral kernels in $\dbar$ problems, we refer the reader to the book by Range \cite{Ran86}.  
	
More recently, using the kernel method Gong \cite{Gong19} constructed a $\dbar$ solution operator on strictly pseudoconvex domains with $C^2$ boundary which gains $1/2$ derivative for any $\var \in \La^r(\ov\Om)$, $r>1$. Assuming $C^2$ boundary still, the authors in a recent preprint \cite{S-Y21-1} found a $\dbar$ solution operator which gains $1/2$ derivative for any $\var \in H^{s,p}(\ov \Om)$, with  $1<p<\infty$ and $s> \frac{1}{p}$; it was further shown in \cite{S-Y21-1} that the same solution operator gains $1/2$ derivative for any $ \var \in \La^r (\ov \Om)$, $r>0$. 

Historically, the problem for constructing a $\dbar$ solution operator with
derivative estimates has been a quite difficult one. In fact, it is still an open problem whether the solution operator of Henkin-Ramanov gains $1/2$ derivative for $\var \in C^k$(or $\La^k$), $ k \geq 1$. In this regard, the use of extension operator and the commutator to prove derivative estimates turns out to be extremely fruitful and arguably essential. We remark that the construction of  
the homotopy operator \re{Hvar_intro} was inspired by the earlier work of Lieb-Range \cite{L-R80}, Peters \cite{Pe91} and Michel-Shaw \cite{M-S99}. Similar techniques have yielded sharp estimates for $\dbar$ equation on other types of domains such as finite type convex domains \cite{Ale06}, and strongly $\C$-linearly convex domains of class $C^{1,1}$ \cite{G-L21}.

The paper is organized as follows. In Section \ref{Section2} we first review the definitions and basic properties of Sobolev spaces. 
We recall the key elements in Rychkov's construction of the universal extension operator, which are used in the proof of \rp{Prop::anti-der_intro}. In Section \ref{Section3}, we construct the anti-derivative operators $\Sc^{m, \gm}_{\Om}$ and prove \rp{Prop::anti-der_intro} based on results in Section \ref{Section2}. In Section \ref{Section4}, we prove Theorems \ref{Thm::Ck_bdy_intro} and \ref{Thm::Cinfty_bdy_intro}. 

We denote the set of positive integers by $\Z^+$, and we let $\Nb= \Z^+ \cup \{ 0 \}$. We shall use $x \lesssim y$ to mean that $x \leq C y$ where $C$ is a constant independent of $x,y$ and $x \approx y$ for $x \lesssim y$ and $y \lesssim x$.  
    
\begin{ack}
  The authors would like to thank Xianghong Gong for many helpful discussions.   
  Part of the paper was prepared during the first-named author's visit to University of Wisconsin-Madison, which was partially supported by the Simons collaboration grant (award number: 505027).
\end{ack}
	\section{Function Spaces}\label{Section2}
	In this section we review the definition and properties of Sobolev spaces. We also recall the Littlewood-Paley characterization, which plays a central role in later proofs.   
	
	\begin{defn}
		Let $\Omega\subseteq\R^N$ be an open set, and let $k\in\N$, $1\leq p<\infty$. We denote by $W^{k,p}(\Om)$ the space of (complex-valued) functions $f\in L^p(\Omega)$ such that $D^\alpha f\in L^p(\Omega)$ for all $|\alpha|\leq k$, and equipped with the norm
		\begin{gather*}
		\|f\|_{W^{k,p}(\Omega)}:=\sum_{|\alpha|\leq k}\|D^\alpha f\|_{L^p(\Omega)}=\sum_{|\alpha|\leq k}\left(\int_\Omega|D^\alpha f (x)|^p\right)^\frac1p \, dV(x),\quad 1\le p<\infty.
		\end{gather*}
\end{defn} 	
% 		Let $\lambda$ be a positive continuous function on $\Omega$; we define the weighted Sobolev space $W^{k,p}(\Omega;\lambda)$ as the space of $f$ in $W^{k,p}_\loc(\Omega)$ such that the following norm is finite:
% 		\begin{gather*}
% 		\sum_{|\alpha|\leq k}\left(\int_\Omega|\lambda(x)D^\alpha f (x) |^p \, dV(x)\right)^\frac1p ,\quad 1\le p<\infty.
% 		\end{gather*}
% 		If $\Omega$ is a domain in $\C^n$ with complex variable $z$, we write instead  $\lambda(z)^p dV(z)$. 
% 	\end{defn}
% 	In our application, we will take $\lambda(x)=\dist(x,b\Omega)^s$ for some $s\in\R$.
	
	We denote by $\Ss(\R^N)$ the space of Schwartz functions, and by $\Ss'(\R^N)$ the space of tempered distributions. For $g \in \Ss(\R^N)$, we set the Fourier transform $\hht g(\xi)=\int_{\R^N} g(x)e^{-2\pi i x \cdot \xi}dx$, and the definition extends naturally to tempered distributions. 
	
	\begin{defn}
		We let $\Ss_0(\R^N)$ denote the space\footnote{In some literature like \cite[Section 5.1.2]{Tri83}, the notation is $Z(\R^N)$.} of all infinite order moment vanishing Schwartz functions. That is, all $f\in\Ss(\R^N)$ such that $\int x^\alpha f(x)dx=0$ for all $\alpha\in\N^N$, or equivalently, all $f\in\Ss(\R^N)$ such that $\widehat f(\xi)=O(|\xi|^\infty)$ as $\xi\to0$. 
	\end{defn}
	
	\begin{defn}[Fractional Sobolev Space] \label{Def:: Sobolev}
		Let $s\in\R$, $1<p<\infty$. We define $H^{s,p}(\R^N)$ to be the fractional Sobolev space consisting of all (complex-valued) tempered distribution $f\in\Ss'(\R^N)$ such that $(I-\Delta)^\frac s2f\in L^p(\R^N)$, and equipped with norm 
		\[ 
		\|f\|_{H^{s,p}(\R^N)} := \|(I-\Delta)^\frac s2f\|_{L^p(\R^N)}.
		\]
		Here $(I-\Delta)^\frac s2 $ is the Bessel potential operator given by 
		\begin{equation}\label{BesselPotent}
		    (I - \Del)^\frac s2 f=((1 + 4\pi^2|\xi|^2)^\frac s2 \hht f(\xi))^\vee.
		\end{equation}
		
	\end{defn}
 
	\begin{rem}
		There is another type of commonly-used fractional Sobolev spaces called \emph{Sobolev-Slobodeckij spaces}, which is in fact the Besov $\Bs_{pp}^s$-spaces. We do not use the Sobolev-Slobodeckij type space in our paper.
	\end{rem}
	
	\begin{defn}\label{domain_def}
		A \textit{special Lipschitz domain} $\omega$ is an open set of $\R^N$ with Lipschitz boundary that has the following form
		\[  
		\omega= \{(x',x_N)\in \R^N: x_N>\rho(x')\},
		\quad x'=(x_1, \cdots, x_{N-1}). 
		\] 
		where $\rho:\R^{N-1}\to\R$ is Lipschitz and satisfies $\|\nabla\rho\|_{L^\infty}<1$.
		
		A \textit{bounded Lipschitz domain} $\Omega$ is a bounded open set of $\R^N$ such that for any $p\in b\Omega $ there is an affine linear transformation $\Phi:\R^N\to\R^N$ and a special Lipschitz domain $\omega$ such that
		\begin{equation*}
		\Omega\cap\Phi(\B^N)=\Phi(\omega\cap\B^N).
		\end{equation*} 
	\end{defn} 

	We denote the positive cone $\Kb$ as
	\begin{equation*}
	\Kb=\{(x',x_N):x_N>|x'|\}.
	\end{equation*} 
	By the assumption $|\na \rho| \leq 1$, we have $\om + \Kb \subseteq \om$. 
	\begin{rem}
		In many literature like \cite[Section 1.11]{Tri06}, the definition for a special Lipschitz domain only requires $\rho$ to be a Lipschitz function. In other words,  $\|\nabla\rho\|_{L^\infty}$ is  finite but can be arbitrary large.
		By taking invertible a linear transformation we can make $\nabla\rho$ small in the new coordinates. 
	\end{rem}

	For functions and distributions defined on domains of $\R^N$, we have the following definition: 
	\begin{defn}\label{domdistdef}
		Let $\Om \subset \R^N$ be an open set.
		\begin{enumerate}[(i)]
    		    \item Define $\Ss' (\Om): = \{\tilde f|_{\Om}:\tilde f\in \Ss' (\R^N) \}$.
    		    \item For $s \in \R$ and $1 < p < \infty$, define $H^{s,p}(\Om): = \{\tilde f|_{\Om}:\tilde f\in H^{s,p} (\R^N)\}$ with norm
		\[
		\| f \|_{H^{s,p}(\Om)} := \inf_{\wti{f}|_{\Om} = f} \|\tilde f\|_{H^{s,p} (\R^N)}.  
		\] 
		\item For $s \in \R$ and $1 < p < \infty$, define $H^{s,p}_0 (\Om)$ to be the subspace of $H^{s,p} (\R^N) $ which is the completion of $C^{\infty}_0 (\Om)$ under the norm $\| \cdot \|_{H^{s,p} (\R^N)}$. 
		    
		\end{enumerate}
	\end{defn} 
	
	\begin{rem}  
		When $k$ is a non-negative integer we have $H^{k,p}(\R^N) = W^{k,p} (\R^N)$ with equivalent norms. If $\Om$ is a bounded Lipschitz domain, we have $H^{k,p}(\Om) = W^{k,p} (\Om)$ which is intrinsically defined. (See \cite[Theorem 2.5.6(ii)]{Tri83} and \cite[Theorem 1.222(i)]{Tri06}.) For non-integer $s$, there is no such simple intrinsic characterization for  $H^{s,p}(\Om)$.  
		For $1<p<\infty$ and $s\in \R $, the Sobolev space $H^{s,p}(\R^N)$ is the same as the Triebel-Lizorkin space $\Fs_{p2}^s(\R^N)$ with equivalent norm. See \cite[Definition 2.3.1/2 and Theorem 2.5.6(i)]{Tri83}. More specifically, we have the following result. 
	\end{rem}
	
	\begin{prop}[Littlewood-Paley Theorem]\label{Prop::L-P}
		Let $\la_0\in\Ss(\R^N)$ be a Schwartz function whose Fourier transform satisfies
		\begin{equation*} 
		\supp\widehat\la_0 \seq \{|\xi|<2\},\qquad\widehat\la_0|_{\{|\xi|\le1\}}\equiv1.
		\end{equation*} 
		For $j\ge1$, let $\la_j$ be the Schwartz function whose Fourier transform is  
		$\hht{\la}_0(2^{-j} \xi) - \hht{\la}_0(2^{-(j-1)} \xi)$.   
		Then for $s\in\R$ and $1<p<\infty$, there exists a $C=C_{\la_0,p,s}>0$ such that
		\begin{equation}\label{Eqn::LPNorm}
		C^{-1}\| f \|_{H^{s,p}(\R^N)} 
		\leq \bigg(\int_{\R^N}\Big(\sum_{j=0}^\infty2^{2js}|\la_j\ast f(x)|^2\Big)^\frac p2dx\bigg)^\frac1p
		\leq C \|f\|_{H^{s,p}(\R^N)},\quad\forall f\in\Ss'(\R^N), 
		\end{equation}
		provided that either term in the inequality is finite. 
	\end{prop} 
	The Triebel-Lizorkin norm $\Fs^{s}_{p2}(\R^N)$ is given by 
	\[
	  \| f \|_{\Fs^{s}_{p2}(\R^N)}:= \bigg(\int_{\R^N}\Big(\sum_{j=0}^\infty2^{2js}|\la_j\ast f(x)|^2\Big)^\frac p2dx\bigg)^\frac1p, 
	\]
	where different choices of $\la$ satisfying the aforementioned properties give equivalent norms (see \cite[Section 2.3.2]{Tri83}). Thus \re{Eqn::LPNorm} implies $H^{s,p}(\R^N)= \Fs^s_{p2}(\R^N)$.

	By way of Definition \ref{domdistdef}, one can also define for an arbitrary open set $\Omega\subset\R^N$ the space $\Fs_{p2}^s(\Omega)=\{\tilde f|_\Omega:\tilde f\in \Fs_{p2}^s(\R^N)\}$ with norm  $\|f\|_{\Fs_{p2}^s(\Omega)} := \inf\limits_{\tilde f|_\Omega=f}\|\tilde f\|_{\Fs_{p2}^s(\R^N)}$ (see \cite[Definition 1.95(i)]{Tri06}.

	% Therefore we have $H^{s,p}(\Omega)=\Fs_{p2}^s(\Omega)$ with
	% 	\begin{equation}\label{Eqn::Hsp=Fp2sOnDomain}
	% 	\|f\|_{H^{s,p}(\Omega)}=\inf\limits_{\tilde f|_\Omega=f}\|\tilde f\|_{H^{s,p}(\R^n)}\approx_{\Omega,s,p}\inf\limits_{\tilde f|_\Omega=f}\|\tilde f\|_{\Fs_{p2}^s(\R^n)}=\|f\|_{\Fs_{p2}^s(\R^n)},\quad\forall f\in \Ss'(\R^n),
	% 	\end{equation}
	% 	provided that either term is finite.

	\begin{prop} \label{Prop::DualSpace} 
		Let $\Omega\seq\R^N$ be a bounded Lipschitz domain. Suppose $1<p<\infty$ and $s\in\R$. Then we have the following equalities of spaces, where the norms are equivalent.
		\begin{enumerate}[(i)]
			\item\label{Item::DualSpace::RN}  $H^{s,p}(\R^N)=H^{-s,p'}(\R^N)'$, where $p'=\frac p{p-1}$.
			\item \label{H_0space} 
			For $s > \frac{1}{p}-1$, $H^{s,p}_0( \Om) 
			= \{f \in H^{s,p}(\R^N): f|_{\overline\Omega^c}  = 0 \}$.  
			\item\label{Item::DualSpace::H0Char} $H^{s,p}_0(\Omega)
			=H^{-s,p'}(\Omega)'$ and $H^{-s,p'}(\Omega)=H^{s,p}_0(\Omega)'
			$, provided that $s > \frac{1}{p}-1$.
			%\item\label{Item::DualSpace::Omega}  $=\{f\in H^{s,p}(\R^N):\supp f\subseteq\overline\Omega\}$, provided that $s>\frac1p-1$. 
			%\item\label{Item::DualSpace::H0Char} $H^{s,p}_0(\Omega)=H^{-s,p'}(\Omega)'\{f\in H^{s,p}(\R^N):\supp f\subseteq\overline\Omega\}$, provided that $s>\frac1p-1$.
		\end{enumerate}
	\end{prop}
	The proof can be done using various results in Triebel \cite{Tri02,Tri06}. The reader can refer to \cite[Proposition 2.11]{S-Y21-1} for details.

	\begin{defn}\label{Defn::DyaRes}
		A \emph{regular dyadic resolution} is a sequence $\phi=(\phi_j)_{j=0}^\infty$ of Schwartz functions, denoted by $\phi\in \mf{R}$, such that
		\begin{itemize}
		\item $\int\phi_0=1$, $\int x^\alpha\phi_0(x)dx=0$ for all $\alpha\in\N^N\backslash\{0\}$.
		\item $\phi_j(x)=2^{Nj}\phi_0(2^jx)-2^{N(j-1)}\phi_0(2^{j-1}x)$, for $j\ge 1$.
		\end{itemize}
		
		A \emph{generalized dyadic resolution} is a sequence $\psi=(\psi_j)_{j=0}^\infty$ of Schwartz functions, denoted by $\psi\in\mf{G }$, such that
		\begin{itemize}
			\item  $\int x^\alpha\psi_1(x)dx=0$ for all $\alpha\in\N^N$.
			\item $\psi_j(x)=2^{N(j-1)}\psi_1(2^{j-1}x)$, for $j\ge1$.
		\end{itemize}
		Here $\psi_0$ can be an arbitrary Schwartz function.
	\end{defn}
	
	\begin{lemma} \label{Lem::K-pair}
		\begin{enumerate}[(i)]
			\item  There is a $g \in \Ss(\R)$ such that $\supp g\subseteq [1,  \infty)$, $\int_{\R} g = 1$ and $\int_{\R} t^k g(t) \, dt =0 $ for all $k \in \Z^+$. 
			\item There exist $\phi \in \Rf$ and $\psi \in \Gf$ such that 
			\begin{itemize} 
				\item 
				$\supp \phi_j$, $\supp \psi_j \subset - \Kb  \cap \{ x_N < -2^{-j} \}$ for all $j \geq 0$. 
				\item
				$\sum_{j=0} \psi_j \ast \phi_j = \del_0 $ is the Dirac delta measure. 
			\end{itemize}            
		\end{enumerate}
	\end{lemma}		
	See \cite[Theorem 4.1(i) and Proposition 2.1]{Ry99}. 
	
	\begin{defn}
		We call $(\phi, \psi)$ which satisfies \rl{Lem::K-pair} a $\Kb$ dyadic pair. 
	\end{defn}

	% For a general bounded Lipschitz domain $\Om$, we can define the extension operator of $\Om$ via patches of special type domain. Take a finite open covering $\bigcup_{j=1}^M U_j \supset b\Om$ such that each $1\le j\le M$ we have an invertible affine linear transform $\Phi_j$ and a special type domain $\Om'_j$ such that $U_j \cap \Om:= U_j\cap\Phi_j(\Om'_j)$. For each $1\le j\le M$ we define $E_jf:=(E_{\Phi_j(\Om'_j)}[f\circ\Phi_j])\circ\Phi_j^{-1}$ where $E_{\Phi_j(\Om'_j)}$ is in \eqref{Eqn::UniExt::HalfPlane}.  Take $\chi_0\in C_c^\infty(\omega)$ and take $\chi_j\in C_c^\infty(U_j)$ for $1\le j\le M$ such that $\sum_{j=0}^M \chi_j^2|_\omega = 1$. We then define 
	% \begin{equation}\label{Eqn::UniExt::BddDomain}
	%     \Ec_{\Om} f  =\chi_0^2f+ \sum_{j=1}^M  \chi_j E_{\Om'_j} (\chi_j f ). 
	% \end{equation}

	\begin{prop} \label{ext_op}
		Let $\Om$ be a bounded Lipschitz domain. Then there exists an extension operator $\Ec$ such that $\Ec: H^{s,p}(\Om) \to H^{s,p} (\R^N)$ is a bounded operator for all $s \in \R$ and $1<p< \infty$. 
	\end{prop} 
	See \cite[Theorem 4.1(ii)]{Ry99}. 
	
	\begin{prop} \label{Prop::C3} 
		Let $\Om$ be a bounded Lipschitz domain. Denote $\delta(x)=\dist(x,b\Omega)$. 
		Then for any $s \geq 0 $ and $1<p<\infty$ there is a $C=C_{s,p,\Omega}>0$ such that 
		$\| \delta^{-s}f\|_{L^p(\Omega)}\leq C  \|f\|_{H^{s,p}_0(\Om)}$. 
	\end{prop}	
	%We are going to prove that $E_{\omega}$ has some smoothing effect in $\omega^c$. The result can somehow recover Proposition \ref{Prop::ext_op} \ref{Item::ext_op::Bdd}. See Theorem ?? and Remark ??.
{This is a standard result. For example, the reader may refer to \cite[Proposition 5.6]{S-Y21-1} for a proof. The strategy is to first prove the statement for integer $s$, and then use complex interpolation for non-integer $s$. We note that a similar result was proved in \cite[Theorem 1.4.4.4]{GrisBook} for Sobolev-Slobodeckij space $W^{s,p}=\Bs_{pp}^s$. The proof for \rp{Prop::C3} when $\Om$ is smooth can be found in \cite[Chapter 5.8]{TriebelStructureFunctions}.}

	\section{The Anti-derivative Operator}\label{Section3}
	In this section we construct the ``anti-derivative'' operators that make a distribution more regular in the Sobolev space, while preserving the vanishing condition on bounded Lipschitz domains. The proof uses an estimate for the Peetre's maximal operators. 
	\begin{defn} 
		Let $M > 0$ and let $\eta = (\eta_j)_{j=0}^{\infty}$ be a sequence of Schwartz functions. The associated \emph{Peetre maximal operators} $\Pc_j^{\eta, M}$ to $\eta$ is given by
		\[
		\Pc^{\eta, M}_j f (x) := \sup_{y \in \R^N} \frac{|(\eta_j \ast f) (x-y)|}{(1+2^j |y|)^M}, \quad f \in \Ss' (\R^N), \quad x \in \R^N. 
		\]  
	\end{defn} 
	
	\begin{prop}[Peetre's maximal estimate] \label{Peetre} 
		Let $0 < p < \infty$. Let $s \in \R$, $M >N / \min(p,2)$, and $\eta = (\eta_j)_{j=0}^{\infty} \in \Gf $. Then there is a $C= C_{p, q,s,M, \eta} >0$ such that 
		\begin{equation}
		\| (2^{js} \Pc_j^{\eta, M} f )_{j=0}^{\infty} \|_{L^p(\ell^2)}
		\leq C \|  f \|_{H^{s,p}(\R^N)}.   
		\end{equation}
	\end{prop} 
	
	\begin{proof} See \cite{Peetre} (and  \cite{PeetreCorrection} for a correction) when $\eta$ is taken to be a ``homogeneous Littlewood-Paley family''. In our case we consider the  ``non-homogeneous case'' and the results can be found in, for example \cite[Theorem 5.1(i)]{BPT}. 
	\end{proof}
	
	\begin{lemma}[Heideman type estimate]\label{Heideman}
		Let $\phi,\psi\in\Gf$ be two generalized dyadic resolutions on $\R^N$. Then for any $M,J>0$ there is a $C=C_{\phi,\psi, M, J}>0$ such that
		\begin{equation*} 
		\int|\phi_j\ast\psi_k(x)|(1+2^{\max(j,k)}|x|)^J dx\le C 2^{-M|j-k|}, \quad \forall \, j, k \in \mathbb{N}.
		\end{equation*}
	\end{lemma} 
	\begin{proof}
		This is a special case of \cite[Lemma 2.1]{B-P-T96}.  
	\end{proof}

	\begin{prop} \label{T_prop}  
		Let $r\in\R$. Fix $\phi,\psi \in \Gf$. We define the operator $T^{r}$ by 
		\begin{equation}\label{Eqn::ExtSmooth::DefT}
		T^{r}f:=\sum_{j=0}^\infty 2^{jr}\psi_j \ast \phi_j \ast f, \quad f \in \Ss. 
		\end{equation}
		Then for any $s \in \R$, $1 \leq p < \infty$, there is a $C = C(\phi, \psi, p,r, s) $ such that 
		\begin{equation}\label{T_est}
		\| T^r f \|_{H^{s-r,p} (\R^N)} \leq C  \|  f \|_{H^{s,p}(\R^N)}. 
		\end{equation}  
	\end{prop} 
	\begin{proof}
		Using the Littlewood-Paley characterization of Sobolev norm, we can write \re{T_est} as 
		\begin{equation} \label{T_est2}
		\left\| \left( 2^{j (s-r) } \la_j \ast \sum_{k=0}^{\infty} 2^{kr} \psi_k \ast \phi_k \ast f  \right)_{j=0}^{\infty} \right\|_{L^p (\ell^2)} 
		\leq C \| ( 2^{js} \la_j \ast f )_{j=0}^{\infty} \|_{L^p(\ell^2)}. 
		\end{equation}
		Fixing $j$ on the left-hand side of \re{T_est2}, we apply \rl{Heideman} to get 
		\begin{align*} 
		2^{j(s-r) } 2^{kr} | \la_j \ast \psi_k \ast \phi_k  \ast f (x) | 
		&\leq 2^{j(s-r)} 2^{kr} \int_{\R^N} |\la_j \ast  \psi_k (x-y) | | \phi_k \ast f (y)| \, dy
		\\ &\leq 2^{j(s-r)} 2^{kr} \int_{\R^N}  |\la_j \ast  \psi_k (x-y) | (1+ 2^k |x-y|)^M  
		\sup_{t \in \R^N} \frac{ | \phi_k \ast  f (t) |}{(1+ 2^{k} | x-t |)^M } \, dy
		\\&\lesssim 2^{j(s-r)} 2^{kr}  2^{-M |j-k|} ( \Pc_k^{\phi, M} f ) (x)
		\\&=2^{k (s-r)} 2^{(j-k) ( s-r) } 2^{kr} 2^{-M|j-k| } ( \Pc_k^{\phi, M} f ) (x)
		\\&= 2^{ks} 2^{(j-k) (s-r)} 2^{-M|j-k|} ( \Pc_k^{\phi, M} f ) (x)
		\leq 2^{-|j-k| }  2^{ks} ( \Pc_k^{\phi, M} f ) (x), 
		\end{align*} 
		where in the last inequality we choose $M=| s-r|+ 1$. 
		It follows that 
		\begin{align} \label{T_est_LPjterm}  
		2^{j(s-r)} | \la_j \ast T^r f (x) | 
		&\leq \sum_{k=0}^{\infty} 2^{j (s-r)}   2^{kr} | \la_j \ast \psi_k \ast \phi_k  \ast f (x) | 
		\\ \nn &\leq \sum_{k=0}^{\infty} 2^{-|j-k|} 2^{ks} (\Pc_k^{\phi, M} f)(x). 
		\end{align}
		Applying Young's inequality and \rp{Peetre} with $M>N$ we have 
		\begin{align} \label{T_est_LPjterm2} 
		\left\| \left( \sum_{k=0}^{\infty} 2^{-|j-k|} 2^{ks} (\Pc_k^{\phi, M} f)  \right)_{j=0}^{\infty} \right\|_{L^p (\ell^2)} 
		&\lesssim \left\| \left(2^{ks} \Pc_k^{\phi, M} f \right)_{k=0}^{\infty} \right\|_{L^p(\ell^2)} 
		\lesssim  \|  f  \|_{H^{s,p} (\R^N)}. 
		\end{align} 
		Combinining \re{T_est_LPjterm} and \re{T_est_LPjterm2} we obtain \re{T_est2}.
	\end{proof}

	\begin{lemma}\label{1d_antider} 
		Suppose $u \in \Ss_0(\R)$ is supported in $\R^+$ and $\int_{\R} u=0$. Then $(2\pi i\xi)^{-1}\widehat u(\xi)$ is the Fourier transform of the function $v(x):=\int_{-\infty}^x u(s)ds=\int_{\R_+}u(x-s)ds$. In addition, $v \in \Ss_0( \R)$ and $v$ is also supported in $\R^+$. 
	\end{lemma}
	\begin{proof}
		By the assumption, $\hht{u}$ has infinite order vanishing at $0$ and rapid decay at $\infty$, thus $(2 \pi i \xi)^{-1} \hht{u} \in \Ss_0 (\R^N)$. 
		
		It suffices to show that $\widehat v(\xi)=(2\pi i\xi)^{-1}\widehat u(\xi)$. Indeed since $\big(\operatorname{p.v.}(2\pi i\xi)^{-1}\big)^\vee(x)=\frac12\operatorname{sgn}(x)$, we have
		\begin{align*}
		\big((2\pi i\xi)^{-1}\widehat u(\xi)\big)^\vee(x) 
		&= \frac12\operatorname{sgn}\ast f (t) 
		\\&=\frac12 \int_{\R_+}u(t -s)ds - \frac12\int_{\R_-}u(t -s)ds
		\\ &=\int_{\R_+}u(t -s)ds-\frac12\int_\R u(t -s)ds
		\\ &=\int_{\R_+}u(t -s) ds= v(t), 
		\end{align*}
		where we use the fact that $\int_\R u(t-s)ds=\int_\R u(s)ds =0$ for all $x\in\R$. 
		Finally since $\supp u \subset \R^+$, we have $\supp u \seq [\ve_0, \infty)$ for some $\ve_0 >0$. Therefore $v(t) = \int_{\R^+} u(t-s) \, ds = 0$ when $t < \ve_0$, in particular $\supp v \in \R^+$.  
	\end{proof} 
For $u,v$ given in \rl{1d_antider}, we shall use the notation: 
$v= D_t^{-1} u$. 
	\begin{prop} \label{antider} 
		Let $K \subset \R^N$ be an open convex cone centered at $0$, and suppose $ f\in \Ss_0(\R^N)$ is supported in $K$. Then there exist $\si_1, \dots \si_N \in \Ss_0(\R^N)$ which satisfy the following:
		\begin{enumerate}[(i)]
			\item  \label{antider_1}
			For each $l =1, \cdots, N$, $\supp \si_l$ is supported in $K$.  
			
			\item \label{antider_2}
			$f = \sum_{l=1}^N D_{x_l} \si_l $. 
		\end{enumerate}
	\end{prop}
	
	\begin{proof} 
		By passing to an invertible linear transformation and taking  suitable linear combinations, we can assume that  $K\supseteq(0,\infty)^N$. Thus by convexity we have $K+[0,\infty)^N\subseteq K$.
		
		By taking Fourier transform, it suffices to find $\si = (\si_1, \cdots, \si_N)$ such that 
		\eq{thh_Fsum} 
		\hht{f} (\xi) = \sum_{l=1}^N ( 2 \pi i \xi_l) \hht{\si}_l (\xi).  
		\eeq  
		Let $g(t)$ be the function as in \rl{Lem::K-pair}, namely that $g \in \mathscr S (\R)$, $\supp g \seq [1, \infty)$, $\int g = 1 $ and $ \int_{\R} t^k g(t) = 0$ for all $k \in \Z^+$. Define
		\[
		g_j (t):= 2^j g(2^j t) - 2^{j-1} g (2^{j-1} t), \quad  j \in \Z, 
		\]
		so that $g_j$ are Schwartz functions supported in $\R^+$. By the moment condition on $g$, we have $g_j \in \Ss_0(\R)$. It follows from  \rl{1d_antider} that $D_t^{-1} g_j \in \Ss_0(\R)$ and $\supp D_t^{-1} g_j \subset \R^+$.

		For $l=1,\dots,N$ and $j\in\Z$, we define $\lambda_{l,j}\in\Ss(\R^N)$ by
		\begin{equation}\label{laljhat} 
		\la_{l, j}(x) 
		:= \left( \prod_{1\leq i \leq l-1} 2^j g(2^j x_i) \right) 
		(D_t^{-1}g_j) (x_l) 
		\left( \prod_{l+1 \leq i \leq N} 2^{j-1} g(2^{j-1} x_i) \right). 
		\end{equation}
		The corresponding Fourier transform satisfies 
		\begin{align} \label{lalj}
		\hht{\la} _{l,j} ( \xi) 
		= \left( \prod_{1 \leq i \leq l-1} \hht{g} (2^{-j} \xi_i)  \right) 
		\left[ \frac{\hht{g} (2^{-j} \xi_l)  - \hht{g} (2^{1-j} \xi_l )}{ 2 \pi i \xi_l}  \right] 
		\left(  \prod_{l+1 \leq i \leq N} \hht{g} (2^{1-j} \xi_{i}) \right). 
		\end{align} 
		By the support conditions	 for $g$ and $D_t^{-1} g_j$, we have  $\supp\lambda_{l,j}\subseteq [2^{-j},\infty)^N$ for all $1\le l\le N$ and $j\in\Z$. 
		Also, we have the scaling relation: 
		\begin{equation} \label{la_scaling} 
		\hht{\la}_{l,j} (\xi) = 2^{-j} \hht{\la}_{l,0} (2^{-j} \xi) ,
		\quad 	\la_{l,j} (x) = 2^{-j} 2^{jN} \la_{l,0} (2^{j} x). 
		\end{equation} 
		
		We now define formally
		\begin{equation}\label{Eqn::anti_deri_psi::DefMu}
		\si_l (x) := \sum_{j \in \Z} \la_{l,j} \ast f (x), \quad 1 \leq j \leq N, 
		\end{equation} 
		and accordingly 
		\begin{equation} \label{mu_l}
		\widehat \si_l(\xi) = \sum_{j\in\Z} \widehat \la_{l,j}(\xi) \hht{f}(\xi),\quad 1\le l\le N. 
		\end{equation} 
		Since $\supp \la_{l,j}  \seq [0,\infty)^N$, we have $\supp \si_l \seq \ov{\bigcup_{j \in \Z} \supp \la_{l,j} }+K  \seq K$. 
		This proves \ref{antider_1}.  
		
		To show that $\si_l \in \Ss_0(\R^N)$, it suffices to show that the sum 
		\[
		\sum_{j \in \Z} \hht{\la}_{l, j} (\xi) \hht{f} (\xi), \quad  1 \leq l \leq N 
		\]
		converges in $\Ss(\R^N)$ and has infinite order vanishing at $\xi = 0$. To this end, we show that for any $M \in \Z$ and $\all, \beta \in \N^N$, the following series converges to a bounded function uniformly for $\xi \in \R^N$: 
		\begin{equation} \label{S0_cvg} 
		| \xi |^M \sum_{j \in \Z} | \pa^{\all}  \hht{\la}_{l,j}(\xi) | |\pa^{\beta} \hht{f}(\xi)|,
		\quad 1 \leq l \leq N. 
		\end{equation} 
		We separate the sum into $\sum_{j >0}$ and $\sum_{j\leq 0}$. By \re{la_scaling} we have 
		\begin{equation} \label{la_der_scaling}
		\pa^{\all} \hht{\la}_{l,j} (\xi) = 2^{- (1 + |\all|) j} \pa^{\all}  (\hht{\la}_{l,0} ) (2^{-j} \xi). 
		\end{equation}
		Since $\hht{f}$ has rapid decay as $\xi \to \infty$ and is flat at $\xi =0$, we have  
		$\sup_{\xi \in \R^N}$ $|\xi|^M |\pa^{\beta} \hht{f} (\xi) |< \infty$. Thus  
		\[
		\sum_{j=1}^{\infty} |\xi|^M |\pa^{\all} \hht{\la}_{l,j} (\xi) | | \pa^{\beta} \hht{f} (\xi) | 
		\leq \| \pa^{\all} \hht{\la}_{l,0} \|_{C^0} \| |\xi|^M \pa^{\beta} \hht{f} \|_{C^0} 
		\sum_{j=1}^{\infty} 2^{- (1+ |\all|) j} < \infty. 
		\]
		On the other hand, by \re{la_der_scaling} and the fact that $\sup_{\xi \in \R^N}  |\xi|^{|\all| + 2} |  \pa^{\all} \hht{\la}_{l,0}(\xi)| = C_0 < \infty$ we have
		\begin{align*}
		| \pa^{\all}\hht{\la} _{l,j} (\xi)|  
		&=  2^{- ( 1 + |\all|)j} |\pa^{\all} \hht{\la}_{l,0} ( 2^{-j} \xi) | 
		\\  &\leq C_0  2^{- ( 1 + |\all|)j} |2^{-j} \xi|^{-|\all| -2}
		\\ &\leq C_0 2^j |\xi|^{-|\all| -2}. 
		\end{align*}
		Setting $ \sup_{\xi \in \R^N}  |\xi|^{M- |\all| -2} |\pa^{\beta} \hht{f} (\xi)| = C_1$, we get 
		\[
		\sum_{j= -\infty}^0 |\xi|^M | \pa^{\all}  \hht{\la}_{l,j} (\xi) | \pa^{\beta} \hht{f} (\xi)| 
		\leq C_0 C_1 \sum_{j= -\infty}^0 2^j < \infty.
		\]
		Thus $\re{S0_cvg}$ converges uniformly for all $M$, and multi-index $\all$ and $\beta$, and we conclude that $\sum_{j \in \Z} \la_{l,j} \ast  f \in \Ss_0 (\R^N)$ for $l = 1, \dots, N$. 
		
		Finally we prove \ref{antider_2}. It suffices to show that for $\xi\in\R^N$
		\begin{equation} \label{antider_Fourier} 
	\hht{f} (\xi) = \sum_{l=1}^N \sum_{j \in \Z} (2 \pi i \xi_l) \hht{\la_{l,j}} (\xi) \hht{f} (\xi). 
		\end{equation}
		Indeed, using \re{laljhat} we have for $\xi \neq 0$, 
		\begin{align*}
  \sum_{l=1}^N  \sum_{j \in \Z} (2 \pi i \xi_l)  \hht{\la}_{l, j}
	&= \sum_{j\in\Z} \sum_{l=1}^N \widehat{g} (2^{-j} \xi_1) \cdots \widehat{g}  (2^{-j} \xi_{l-1})   
		\big(\widehat{g} (2^{-j} \xi_l)  - \widehat{g} (2^{1-j} \xi_l )\big)
		\widehat{g} (2^{1-j} \xi_{l+1}) \cdots \widehat{g}  (2^{ 1-j} \xi_N) 
		\\ &= \lim_{J \to \infty} \sum_{j=-J}^J \widehat g(2^{-j}\xi_1)\cdots \widehat g(2^{-j}\xi_N)-\widehat g(2^{1-j}\xi_1)\cdots\widehat g(2^{1-j}\xi_N) 
		\\
		& =  \lim_{J \to \infty} \widehat g(2^{-J}\xi_1)\cdots \widehat g(2^{-J}\xi_N)-\widehat g(2^{J}\xi_1)\cdots\widehat g(2^{J}\xi_N) = 1,  
		\end{align*} 
		where we used the fact $\hht{g}(0) = 1$ and $\lim_{\xi \to \pm \infty} \hht{g} (\xi) = 0$. 
		This proves \re{antider_Fourier} and \ref{antider_2}.   
	\end{proof}
	
	By induction on the order of derivatives, we easily obtain the following. 
	\begin{prop} \label{hord_antider}  
		Let $K \subset \R^N$ be a cone centered at $0$ and let $g \in \Ss_0 (\R^N )$ be supported in $K$. Then for each $m \in \Z^+$ there exists $\{ \mu_{\beta} \}_{|\beta| = m} \in \Ss_0 (\R^N)$ such that   
		\begin{enumerate}[(i)]
			\item  \label{hord_antider_1}  $\supp \mu_{\beta} \subset K$. 
			\item \label{hord_antider_2}
			 $g = \sum_{|\beta| = m} D^{\beta} \mu_{\beta} $.
		\end{enumerate} 
	\end{prop}
	
	By combining \rp{hord_antider} with \rp{T_prop} we have the following: 
	\begin{prop}\label{special_Lip_anti} 
		For each $m \in \Z^+$, there exist convolution operators $\Fc_0$ and $\Fc_{\beta}$ with $|\beta| = m$, such that
		\begin{enumerate}[(i)]
			\item \label{special_Lip_anti_1}
			$\Fc_0: H^{s,p} (\R^N) \to H^{s+m, p} (\R^N)$, and $\Fc_{\beta}: H^{s,p} (\R^N) \to H^{s+m, p} (\R^N) $ for each $\beta \neq 0$. 
			\item \label{special_Lip_anti_2}
			For any $f \in \Ss'(\R^N)$, $f = \Fc_0 f + \sum_{|\beta| = m} D^{\beta} \Fc_{\beta} f$. 
			\item \label{special_Lip_anti_3} 
			Let $\omega$ be a special Lipschitz domain and suppose $f \in \Ss'(\R^N)$ satisfies $f|_{\om} \equiv 0$. Then $\Fc_{\beta} f|_{\om} \equiv 0$ for all $\beta$.   
		\end{enumerate}
	\end{prop} 
	\begin{proof}
		Let $(\phi,\psi)$ be a $\Kb$-dyadic pair. Since $\psi_1 \in \Ss_0 (\R^N)$ is supported $-\Kb$, in view of \rp{hord_antider} with $g=\psi_1$, there exist $\mu_{\beta} \in \Ss_0(\R^N)$ supported in $-\Kb$ such that 
		\begin{equation} \label{psi1_sum} 
		\psi_1 = \sum_{|\beta| = m} D^{\beta} \mu_{\beta}. 
		\end{equation}
		Define  
		\[
		\eta_0^{\beta} := 0, \quad \eta^{\beta}_j (x) := 2^{N (j-1)} \mu_{\beta}(2^{j-1} x ), \quad j \geq 1. 
		\]
		In particular, we have $ \eta_1^{\beta} (x) = \mu_{\beta}(x)$ and $\eta_j^{\beta} (x) = 2^{N(j-1)} \eta_1^{\beta} (2^{j-1} x)$.  
		Hence $\eta^{\beta} \in \mf{G}$, and $ \eta_j^{\beta} \subset - \Kb$. Setting $g := \psi_1$ in \rp{hord_antider}, we can write for $j \geq 1$, 
		\begin{align*}
		\psi_j(x) &= 2^{N(j-1)} \psi_1 (2^{j-1} x)
		\\ &=  \sum_{| \beta| = m} 2^{N(j-1)} (D^{\beta} \mu_{\beta})(2^{j-1} x)
		\\ &= \sum_{| \beta| = m} 2^{m(1-j) }D^{\beta} \eta_j^{\beta} (x). 
		\end{align*}
		Therefore 
		\begin{align*}
		f &= \sum_{j=0}^{\infty} \psi_j \ast \phi_j \ast f 
		\\ &=  \psi_0 \ast \phi_0 \ast f 
		+ \sum_{|\beta| = m} D^{\beta}  \sum_{j=1}^{\infty} 2^{m(1-j)}  \eta_j^{\beta} \ast \phi_j \ast f   
		\\ &=: \Fc_0 f + \sum_{|\beta| = m}  D^{\beta}( \Fc_{\beta} f), 
		\end{align*}
		where we set 
		\begin{equation} \label{antider_Fbeta} 
		\Fc_0 f := \psi_0 \ast \phi_0 \ast f, \quad 
		\Fc_{\beta} f := \sum_{j=1}^{\infty} 2^{m(1-j)}  \eta_j^{\beta} \ast \phi_j \ast f  , \quad  |\beta| =m. 
		\end{equation}
		Clearly $\Fc_0:H^{s,p}(\R^N)\to H^{s+l,p}(\R^N)$ for any $l \in \Z^+$, (since $\psi_0, \phi_0 \in \Ss(\R^N)$), in particular for $l=m$. By \rp{T_prop}, $\Fc_{\beta}: H^{s,p} (\R^N) \to H^{s+m, p} (\R^N)$. Finally if $\supp  f \subseteq \om^c$, then 
		$\supp (\eta_j \ast \phi_j \ast f) \subseteq \supp \eta_j + \supp \phi_j + \supp f =
		-\Kb - \Kb + \om^c  \subseteq \om^c$, for each $j \geq 0$. By \re{antider_Fbeta} we have $\supp \Fc_{\beta}f \subseteq \om^c$.  
	\end{proof} 
	
	We are ready to prove the main result of the section, namely constructing anti-derivative operators on bounded Lipschitz domains. This is done by using partition of unity and \rp{special_Lip_anti}.   
	
	We now introduce the following notation for partition of unity: 
	\begin{note}\label{Note::PartitionUnity}
		Let $\Omega\subset\R^N$ be a bounded Lipschitz domain, and let $\Uc$ be a bounded open set containing $\overline{\Omega}$. We shall use the following construction, which can be obtained by standard partition of unity argument:
		\begin{itemize}
			\item $(U_\nu)_{\nu=1}^M$ are finitely many bounded open sets on $\R^N$.
			\item$(\Phi_\nu:\R^N\to\R^N)_{\nu=1}^M$ are invertible affine linear transforms.
			\item $(\chi_\nu)_{\nu=0}^M$ as $C_c^\infty$-functions on $\R^N$, $ 0 \leq \chi_{\nu} \leq 1$. 
			\item $(\omega_\nu)_{\nu=1}^M$ as special Lipschitz domains on $\R^N$. That is $\omega_\nu=\{x_N>\rho_\nu(x')\}$ for some $\|\nabla\rho_\nu\|_{L^\infty}<1$.
		\end{itemize}
		They have the following properties: 
		\begin{itemize} 
			\item $b\Omega\subset\bigcup_{\nu=1}^MU_\nu \subset \subset \Uc$.
			\item\label{Item::PartitionUnity::chi} $\chi_0\in C_c^\infty(\Omega)$, $\chi_\nu\in C_c^\infty(U_\nu)$ for $1\le \nu\le M$, and $\chi_0+\sum_{\nu=1}^M\chi_\nu^2\equiv1$ in a neighborhood of $\overline\Omega$.
			\item\label{Item::PartitionUnity::Phi} For each $1\le \nu\le M $, $U_\nu=\Phi_\nu(\B^N)$ and $U_\nu\cap\Omega=U_\nu\cap\Phi_\nu(\omega_\nu)$.
		\end{itemize} 
	\end{note} 
	
	\begin{prop} \label{Thm::Lip_anti}
		Let $\Om$ be a bounded Lipschitz domain.  For any $m \in \Z^+$, there exist operators $\Sc_{\Om}^{m, \gm}: \Ss' (\R^N) \to \Ss'(\R^N)$, with $|\gm| \leq m$ such that  
		\begin{enumerate}[(i)] 
			\item 
			$\Sc^{m, \gm}_{\Om}: H^{s,p} (\R^N)  \to H^{s+m,p} (\R^N)$ for all $1 < p < \infty$ and $s \in \R$. 
			\item 
			\[ 
			f =  \sum_{|\gm| \leq m } D^{\gm} \Sc^{m,\gm}_{\Om} f,  \quad \text{for $ f \in \Ss'(\R^N)$}.    
			\] 
			\item 
			If $f \in \Ss'(\R^N)$ satisfies $f|_{\Om} \equiv 0$, then $\Sc_{\Om}^{m, \gm} f|_{\Om} \equiv 0$ for all $\gm$.   
		\end{enumerate} 
	\end{prop}

	\begin{proof}
		We adopt the Notation  \ref{Note::PartitionUnity}.
		Define $\chi_\infty: = 1- \chi_0 - \sum_{\nu=1}^M \chi_\nu^2$. Let $\mu_\infty\in C^\infty(\R^N)$ be such that 
		\begin{itemize}
			\item $1- \mu_\infty\in C_c^\infty(\R^N)$,\quad $\supp\mu_\infty\subseteq\R^N\sm\overline\Omega$,\quad and $\mu_\infty\chi_\infty\equiv\chi_\infty$.
		\end{itemize}
		For $f \in \Ss' (\R^N)$, we have
		\begin{align}  \label{antider_fsum}
		f &= \chi_0 f + \chi_{\infty} f + \sum_{\nu=1}^M \chi_{\nu}^2 f 
		\\ \nn &= \chi_0 f + \mu_\infty\chi_{\infty} f  + \sum_{\nu=1}^M \chi_{\nu}  \left(  ( \chi_{\nu}f ) \circ \Phi_{\nu} \circ \Phi_{\nu}^{-1}  \right).  
		\end{align}
		By \rp{special_Lip_anti}, there exist operators $\Fc_0: H^{s,p}(\R^N) \to H^{s+m,p} (\R^N)$ as well, $\Fc_{\beta}: H^{s,p}(\R^N) \to H^{s+m,p} (\R^N)$, $|\beta| =m$ (formula \re{antider_Fbeta}) such that 
		\begin{align} \label{antider_fsum2}
		f &=  \Fc_0 (\chi_0 f) + \mu_{\infty}\cdot \Fc_0 (  \chi_{\infty} f) + \sum_{\nu=1}^M \chi_{\nu} \cdot \left( \Fc_0 \left[ (\chi_{\nu} f ) \circ \Phi_{\nu} \right] \circ \Phi_{\nu}^{-1} \right) 
		\\ \nn &\quad+ \sum_{|\beta| =m}  \left( D^{\beta}\Fc_{\beta} (\chi_0 f ) 
		+ \mu_{\infty}\cdot D^{\beta}  \Fc_{\beta}  (  \chi_{\infty} f ) +  
		\sum_{\nu =1}^M  \chi_{\nu} \cdot  \left( D^{\beta} \Fc_{\beta} \left[ (\chi_{\nu }f) \circ \Phi_{\nu} \right]  \circ \Phi_{\nu}^{-1} \right) \right).  
		\end{align} 
		%\blu{Here we use the fact that $\chi_\infty f=\mu_\infty\chi_\infty f=\mu_\infty\cdot \Fc_0(\chi_\infty f)+\sum_{|\beta|=m}\mu_\infty\cdot D^\beta\Fc_\beta(\chi_\infty f)$.}
		
		Set $\wti{\Fc}_{\beta, \nu} (f) := \Fc_{\beta} \left[ (\chi_{\nu }f) \circ \Phi_{\nu} \right]  \circ \Phi_{\nu}^{-1}$, so then $\Fc_{\beta} \left[ (\chi_{\nu }f)  \circ \Phi_{\nu}  \right] = (\wti{\Fc}_{\beta, \nu}f)\circ \Phi_{\nu}$. Since $\Phi_{\nu}$ are linear,
		\begin{equation} \label{Fbeta_der} 
		D^{\beta} \Fc_{\beta} \left[ (\chi_{\nu }f)  \circ \Phi_{\nu} \right]
		= \sum_{| \tau| = m}  c_{\nu,\beta,\tau} [ D^{\tau}( \wti{\Fc}_{\beta, \nu}f)]\circ  \Phi_{\nu}, \quad |\beta| =m 
		\end{equation} 
		for some constants $c_{\nu,\beta,\tau}$. 
		Using \re{Fbeta_der} and the Leibnitz rule we get   
		\begin{align*}
		&\sum_{|\beta| =m} \sum_{\nu =1}^M  \chi_{\nu} \cdot  \left( D^{\beta} \Fc_{\beta} \left[ (\chi_{\nu }f) \circ \Phi_{\nu} \right]  \circ \Phi_{\nu}^{-1} \right)
		\\  &\quad = \sum_{\nu =1}^M  \sum_{|\beta|, | \tau| = m}  c_{\nu,\beta,\tau}  \chi_{\nu} \cdot D^{\tau} \wti{\Fc}_{\beta, \nu}(f)
		\\  & \quad = \sum_{ |\gm| \leq m} D^{\gm} \left( \sum_{\nu=1}^M \sum_{|\beta|=m} h_{\nu, \beta, \gm} \cdot \wti{\Fc}_{\beta, \nu}(f) \right), 
		\end{align*}
		where $ h_{\nu, \beta, \gm}$ are linear combinations of $c_{\nu,\beta,\tau} \chi_{\nu}$ and their derivatives up to order $m$. 
		In particular $h_{\nu, \beta, \gm}$ are $C^{\infty}$ functions supported in $U_{\nu}$. 
		Substituting into \re{antider_fsum2} we obtain    
		\begin{align*}
		f &=  \sum_{ | \gm| \leq m} D^{\gm} \Sc^{m, \gm}_\Om   f,  
		\end{align*}
		where we define
		\begin{gather*}
		\Sc^{m,0}_\Om f :=  \Fc_0 (\chi_0 f) + \mu_{\infty}\cdot \Fc_0 (\chi_{\infty} f) + \sum_{\nu=1}^M \chi_{\nu} \cdot \left( \Fc_0 \left[ (\chi_{\nu} f ) \circ \Phi_{\nu} \right] \circ \Phi_{\nu}^{-1} \right)  
		+ \sum_{\nu=1}^M \sum_{| \beta| = m} h_{\nu, \beta, 0}\cdot \wti{\Fc}_{\beta, \nu} (f)   
		\\ 
		\Sc^{m, \gm}_\Om f: = \sum_{\nu=1}^M \sum_{|\beta|=m} h_{\nu, \beta, \gm}\cdot  \wti{\Fc}_{\beta, \nu}(f),  \quad  1 \leq |\gm| < m, 
		\\  
		\Sc^{m, \gm}_\Om f: = \Fc_{\gm}  ( \chi_0 f) + \mu_{\infty}\cdot \Fc_{\gm} (\chi_{\infty} f) + \sum_{\nu=1}^M \sum_{|\beta|=m} h_{\nu, \beta, \gm} \cdot \wti{\Fc}_{\beta, \nu}(f),  \quad |\gm| =m . 
		\end{gather*} 
		The boundedness of $\Sc_{m, \gm}: H^{s,p}(\R^N) \to H^{s+m,p} (\R^N)$ follows from that of $\Fc_0$ and $\Fc_{\beta}$.  
		Suppose now that $f|_{\Om} \equiv 0$; then $\chi_0 f \equiv 0$ in $\R^N$, and $(\chi_{\nu} f) \circ \Phi_{\nu} \equiv 0$ in $\om_{\nu}$  for each $1 \leq \nu \leq M$.  By \rp{special_Lip_anti} \ref{special_Lip_anti_3},  $\Fc_{0} [(\chi_{\nu} f) \circ \Phi_{\nu}], \Fc_{\beta} [(\chi_{\nu} f) \circ \Phi_{\nu}]  \equiv 0$ in $\om_{\nu}$ for all $|\beta| = m$. It follows that 
		\begin{gather*}
		\Fc_0 (\chi_0 f) := \psi_0 \ast \phi_0 \ast (\chi_0 f) \equiv 0, \quad \text{in $\R^N$}; \quad \mu_{\infty} \cdot\Fc_0 (\chi_{\infty} f), \, \mu_{\infty} \cdot\Fc_{\beta} (\chi_{\infty} f) \equiv 0, \quad \text{in $\Om$};  
		\\  
		\chi_{\nu} \cdot \left( \Fc_0 \left[ (\chi_{\nu} f ) \circ \Phi_{\nu} \right] \circ \Phi_{\nu}^{-1} \right)  \equiv 0, \quad 
		h_{\nu, \beta, 0}\cdot \wti{\Fc}_{\beta, \nu} (f), \, 
		h_{\nu, \beta, \gm} \cdot \wti{\Fc}_{\beta, \nu}(f)  \equiv 0, \quad \text{in $\Om$}. 
		\end{gather*} 
		We conclude that $(\Sc_{m, \gm} f)|_{\Om} \equiv 0$ with $ |\gm| \leq m$. 
	\end{proof}

	\section{Sobolev Estimates of Homotopy Operators}\label{Section4}  
	In this section we prove Theorems \ref{Thm::Ck_bdy_intro} and \ref{Thm::Cinfty_bdy_intro}.
	First we apply our anti-derivative operators from the last section to the commutator. 
	
	\begin{lemma}\label{comm_antider}
		Let $\Omega\subset\C^n$ be a bounded Lipschitz domain, and let $\Ec = \Ec_{\Om}$ be the extension operator given by Proposition \ref{ext_op}. Then for any $l \in \Z^+$, there exists a collection of operators $(\Sc^{l,\gm}_{\Om})_{|\gm| \leq l}$, such that for any $s\in\R$ and $1<p<\infty$, 
		\begin{enumerate}[(i)] 
			\item\label{comm_antider_1} 
			$ \Sc^{l, \gm}_{\Om}:H^{s,p}_{(0,q+1)}(\Om)\to H^{s+l ,p}_{(0,q+1)}(\C^n)$ are bounded linear operators for $1 \le q\le n$. 
			\item\label{comm_antider_2} For every $\var \in H^{s,p}_{(0,q)}(\Om)$, we have $\Sc^{l, \gm}_{\Om} [\dbar, \Ec] \var |_\Omega \equiv 0$, and  \begin{equation}\label{Eqn::DefEcGcForPsiCX::AntiDev}
			[\dbar,\Ec]\varphi= \sum_{ |\gm| \leq l} D^{\gm} \Sc^{l, \gm}_{\Om} [\dbar, \Ec] \var \quad\text{in }\C^n.
			\end{equation}
			Here $D^{\gm} $ acts on $\Sc^{l,\gm}_{\Om} [\dbar, \Ec]\var$ component-wise. 
		\end{enumerate}
	\end{lemma}
	
\begin{proof}
    Since the coordinate components of $[\dbar,\Ec]\varphi$ are the linear combinations of the coordinate components of $[D_1,\Ec]\varphi,\dots,[D_{2n},\Ec]\varphi$, where $D_i = \frac{\partial}{\partial x_i}$, the results follow immediately by  applying \rt{Thm::Lip_anti}.
\end{proof}
Our homotopy operator is given by
	\begin{equation}\label{Hq_def}
	\begin{aligned}
	\Hc_q \varphi(z) :=& \int_{\Uc} K^0_{0, q-1}(z,\cdot) \wedge \Ec \var
	+ (-1)^{|\gm|} \sum_{|\gm| \leq l} \int_{\Uc\sm\ov\Omega}  D^{\gm}( K^{01}_{0, q-1}(z,\cdot)) \wedge \Sc^{l, \gm}_{\Om} [\dbar, \Ec] \var, 
	\end{aligned}
	\end{equation}
Here $\Ec$ and $\Sc^{l, \gm}_{\Om}$ are operators given in Lemma \ref{comm_antider}, and we multiply $\Ec$ and $\Sc^{l, \gm}_{\Om}$ by a suitable smooth cut-off function so that $ \supp (\Ec \var), \supp (\Sc^{l, \gm}_{\Om} [\dbar, \Ec] \var) \subset \Uc$. The kernels $K^0$ and $K^{01}$ are given below by formulae \re{K0} and \re{K01}. The operators $\Ec, D^{\gm}_{\zeta}$ and $\Sc^{l, \gm}$ are applied to each component of $\var$, $K^{01}_{0,q-1}$ and $[\dbar, \Ec] \var $ respectively, and the integrals are interpreted as the sum of the integrals over all component functions. 

We need to interpret formula \re{Hq_def} when the coefficients of $\var$ are tempered distributions in the class $H^{s,p}$ where $s$ is negative. As it turns out, the first integral can be written as a linear combination of the convolutions with kernel $\frac { \ov{z_i}}{|z|^{2n}}$, and therefore makes sense for any tempered distribution. For the second integral, in view of Proposition \ref{Prop::Hq_Ckbdy}, we have $\Sc^{l, \gm}_{\Om} [\dbar, \Ec] \var \in H^{s,p}(\Uc)$ for positive $s$ if $l$ is sufficiently large, in particular $\Sc^{l, \gm}_{\Om} [\dbar, \Ec] \var \in L^1_{\loc}$. Since $K^{01}$ is smooth for fixed $z \in \Om$, the integrand defines a function in $L^1_\loc(\Uc\sm\Omega)$ for $z \in \Om$. 

When the boundary is $C^{k+2}$, we can take $l$ to be any integer greater or equal to $k+1$. Later we will show the $\Hc_q \var$ does not depend on the choice of $l \geq k+1$ (see the remark after \rc{Cor::hf}). 
	
The main result of the section is the following:
	
	\begin{thm}\label{Thm::EstHq}
		Let $\Omega\subset\C^n$ be a bounded strictly pseudoconvex domain with $C^{k+2}$ boundary, for $k \geq 1$. Let $1\le q\le n$ and let $\Hc_q$ be defined by \eqref{Hq_def}, where $l$ is any positive integer greater or equal to $k+1$. Then for any $s>\frac1p - k$ and $1<p<\infty$, $\Hc_q$ defines a bounded linear operator $\Hc_q: H_{(0,q)}^{s,p} (\Om)\to H^{s+\frac12,p}_{ (0,q-1)}(\Om)$.
	\end{thm}

	\rt{Thm::EstHq} allows us to prove a homotopy formula under much weaker regularity assumption and \rt{Thm::Ck_bdy_intro} follows as a result.  
	\begin{cor}[Homotopy formula] \label{Cor::hf} 
		Under the assumptions of \rt{Thm::EstHq}, suppose $\varphi\in H^{s,p}_{0,q}(\Om)$ satisfies $\dbar\varphi\in H^{s,p}_{(0,q+1)}(\Om)$, for $1\le q\le n$ and $s>\frac1p-k$. Then the following homotopy formula holds in the sense of distributions: 
		\begin{equation}\label{hf_eqn}
		\varphi=\dbar \Hc_q\varphi+\Hc_{q+1}\dbar\varphi.
		\end{equation}
		In particular for $\varphi$ in $\Omega$ with $H^{s,p}_{(0,q)}(\Om)$ which is $\dbar$-closed, we have $\dbar \Hc_q \var = \var$ and $\Hc_q \var \in H^{s+\frac12,p}_{ (0,q-1)} (\Om)$. 
	\end{cor}
	\begin{proof}
		For simplicity we will write $\Sc^{l, \gm}$ instead of $\Sc^{l,\gm}_{\Om}$. 
		First we show that \eqref{hf_eqn} is valid for $\varphi \in C^\infty_{(0,q)}(\overline{\Om})$.  
		Indeed, when $\varphi\in C^\infty$, by \rl{comm_antider} \ref{comm_antider_1} $\Sc^{l,\gm} ([\dbar, \Ec] \var) \in C^\infty_{c}(\Uc)$. 
		By \rl{comm_antider} \ref{comm_antider_2},  $\Sc^{l, \gm} ([\dbar, \Ec] \var)|_{\Omega} \equiv 0$ which implies $\Sc^{l,\gm}([\dbar,\Ec] \var)|_{b\Om}\equiv0$ by continuity. 
		Applying integration by parts on the domain $\U\sm\Om$ we get for $z \in \Om$, 
		\begin{align*}
		\Hc_q \varphi(z) :=& \int_{\Uc} K^0_{0, q-1}(z,\cdot) \wedge \Ec \var  + (-1)^{|\gm|} \sum_{|\gm| \leq l} \int_{\Uc \sm\ov \Om} D^{\gm}(K^{01}_{0, q-1}(z,\cdot)) \wedge \Sc^{l,\gm}  ([\dbar,\Ec]\varphi)  
		\\
		=&\int_{\Uc} K^0_{0, q-1}(z,\cdot) \wedge \Ec \var + \int_{\Uc\sm\ov\Omega}  K^{01}_{0, q-1}(z,\cdot) \wedge  \sum_{|\gm| \leq l}D^{\gm} \Sc^{l,\gm}([\dbar,\Ec]\varphi)
		\\
		=&\int_{\Uc} K^0_{0, q-1}(z,\cdot) \wedge \Ec \var + \int_{\Uc\sm\ov\Omega}  K^{01}_{0, q-1}(z,\cdot) \wedge[\dbar,\Ec]\varphi.
		\end{align*}
		By \cite[Proposition 2.1]{Gong19}, 
		$\var = \dbar \Hc_q \var + \Hc_{q+1} (\dbar \var)$. 
		For general $\var \in H^{s,p}_{(0,q)}(\Om)$ such that $\dbar\varphi\in H^{s,p}_{(0, q+1)}(\Om)$, we use approximation. In view of \rp{Prop::Sobolev_smoothing}, we can find a sequence $\var_{\ve}\in C^{\infty}(\ov \Om)$ such that 
		\begin{gather*}
		\var_{\ve} \overset{\ve \to 0}{\ra} \var \quad \text{in $H^{s,p}(\Om)$ },  
		\\ 
		\dbar \var_{\ve} \overset{\ve \to 0}{\ra}  \dbar \var \quad \text{in $H^{s,p}(\Om)$}. 
		\end{gather*} 
		By \rt{Thm::EstHq} we have for any $s> -k + \frac{1}{p}$,  
		\begin{align*} 
		\| \dbar  \Hc_q (\var_{\ve} - \var) \|_{H^{s-\yh, p} (\Om)} 
		&\leq \|  \Hc_q (\var_{\ve} - \var) \|_{H^{s+\yh, p} (\Om)}  
		\\ &\leq \| \var_{\ve} -  \var \|_{H^{s,p}}, 
		\end{align*}
		and also  
		\[
		\| \Hc_{q+1} \dbar (\var_{\ve} - \var) \|_{H^{s+\yh, p}(\Om)} 
		\leq \| \dbar (\var_{\ve}- \var)  \|_{H^{s,p} (\Om)}. 
		\]
		Letting $\ve \to 0$ we get \re{hf_eqn}. 
	\end{proof} 
\begin{rem} \label{Rem::l_ind} 
 In the expression for $\Hc_q$ (see \re{H0H1}), we use $l$ to denote the number of anti-derivatives. In view of \rt{Thm::EstHq}, we can show that $\Hc_q^1$ actually does not depend on $l$, for $l \geq k+1$. This is seen as follows. Denote 
\[
  T_q^l \var:= \sum_{|\gm| \leq l}(-1)^{|\gm|} \int_{\U \sm \ov\Om} D^{\gm}_\zeta K^{01}_{0,q} \we \Sc^{l,  \gm}_{\Om} [\dbar, \Ec ]\var 
\] 
By the proof of \rc{Cor::hf}, we have $T_q^{l_1} \var = T_q^{l_2} \var$ for $\var \in C^{\infty}(\ov \Om)$ for any $l_1, l_2$, and is equal to 
\[
  T \var (z)= \int_{\Uc\sm\ov\Omega}  K^{01}_{0, q-1}(z,\cdot) \wedge[\dbar,\Ec]\varphi, \quad \var \in C^{\infty}(\ov \Om). 
\] 
For general $\var \in H^{s,p}(\Om)$, $s> \frac1p -k$, we can take $\var_{\ve} \in C^{\infty}(\ov{\Om})$ such that $\| \var_\ve - \var \|_{H^{s,p}(\Om)} \to 0$. By \rt{Thm::EstHq}, for any $l \geq k+1$ and $s>\frac1p -k$, $T^l_q$ defines a bounded linear operator from $H^{s,p}(\Om) \to H^{s+\yh, p} (\Om)$. Thus we have
\[
 T^l_q \var = \lim_{\ve \to 0} T^l_q \var_{\ve},  
\] 
where the limit is taken with respect to the norm $H^{s+\yh,p}(\Om)$. Hence $T^{l_1}_q \var = T^{l_2}_q \var$ for all $\var \in H^{s,p}(\Om)$, $s> \frac1p - k $ and any $l_1, l_2 \geq k +1$. 

Accordingly, we can define $\Hc^1_q$ on $H^{s,p}(\Om)$ with $s > \frac1p -k$, by setting 
\[
  \Hc^1_q \var := T^l_q \var, \quad \var \in H^{s,p}_{(0,q)}(\Omega), 
\] 
where $l$ is any positive integer greater or equal to $k+1$.  
\end{rem}
For the remainder of the section we will prove \rt{Thm::EstHq}. 
We adopt the following notation. 
	\begin{gather*}
	\Om_{\ve} = \{ z \in \C^n: \dist(z, \Om) < \ve \} ,    
	\quad \Om_{-\ve} = \{ z \in \Om: \dist(z, b\Om) > \ve \}. 
	\end{gather*} 
When the boundary is only finitely smooth, we need the following regularized Henkin-Ramirez function constructed by Gong \cite{Gong19}. 
\begin{lemma}
  Let $\Om$ be a bounded domain in $\C^{n}$ with $C^{k+2}$ boundary. Let $\rho_{0}$ be a $C^{k+2}$ defining function of $\Om$. That is, there exists a neighborhood $\mc{U}$ of $\ov{\Om}$ such that $\Om = \{ z \in \mc{U}: \rho_{0} < 0 \}$ and $\na \rho_0 \neq 0$ on $b\Om$. Let $\wti{\rho}_0$ be the Whitney extension of $\rho_0$ from $\Om$. Then $\wti{\rho_{0}} \in C^{k+2} (\C^{n}) \cap C^{\infty} (\C^{n} \sm \ov{\Om})$ with $\wti{\rho_0} = \rho_{0} $ in $\ov{\Om}$, and for $0 < \del(x):= \dist(x, \Om) < 1$, we have
\eq{regdfest}
	|D_{x}^{i} \wti{\rho_{0}} (x) | \leq C_{i} \| \rho_0  \|_{C^{k+2}(\ov\Om)} (1+ \delta(x)^{k+2-i} ), 
\eeq
for $i \in \Nb$.  
\end{lemma}
\begin{proof}
  See \cite[Lemma 3.7]{Gong19}.   
\end{proof}
We call $\wti{\rho_{0}}$ the \emph{regularized defining function with respect to $\rho_0$}.  

%\bro{Starting here every $\rho_0$ below is $\widetilde{\rho_0} $ right? Or you should swap $\rho_0$ and $\widetilde{\rho_0}$ in the above lemma}
\begin{prop}[Regularized Henkin-Ramirez function] \label{Prop::Reg_H-R}

Let $\Om \subset \C^n$ be strictly pseudoconvex with $C^{k+2}$ boundary, with $k \geq 0$. Let $\rho_1 = e^{L_{0} \rho_{0} -1}$, where $L_0$ is sufficiently large so that $\rho_1$ is strictly plurisubharmonic in a neighborhood $\om$ of $b\Om$. Let $\rho$ be the regularized defining function with respect to $\rho_1$. Then there exist $\del_0 > 0$, $\ve>0$ and function $W$ (called regularized Leray map) in $\Om_{\ve} \times (\Om_{\ve} \sm \Om_{-\ve})$ satisfying the following. 
		\begin{enumerate}[(i)] 
		    \item \label{Item::Reg_H-R::1} $W: \Om_{\ve} \times (\Om_{\ve} \sm \Om_{-\ve}) \to \C^{n}$ is a $C^{k+1}$ mapping, $W(z, \zeta)$ is holomorphic in $z \in \Om_{\ve}$, and $\Phi (z, \zeta) = W(z, \zeta) \cdot (\zeta -z) \neq 0$ for $\rho(z) < \rho(\zeta)$. 
			\item
			\label{Item::Reg_H-R::2} If $|\zeta -z| < \del_0$, and $\zeta \in \Om_{\ve} \sm \Om_{-\ve} $, then $\Phi(z, \zeta) = F(z, \zeta) P(z, \zeta)$, $P(z, \zeta) \neq 0$ and 
			\[ 
			F(z, \zeta) = - \sum_{j=1}^n \DD{\rho}{\zeta_{j}} (z_{j} - \zeta_{j}) + \sum_{j,k=1}^n a_{jk} (\zeta) (z_{j} - \zeta_{j}) (z_{k} - \zeta_{k}), 				\] 
			\[ 
			Re F(z, \zeta) \geq \rho(\zeta) - \rho(z) + |\zeta -z|^{2} / C, 
			\]
			with $P,F \in C^{k+1}(\Om_{\ve} \times (\Om_{\ve} \sm \Om_{- \ve}))$ and $a_{jk} \in C^{\infty} (\C^{n})$. 
			\item[(iii)] 		\label{Item::Reg_H-R::3}
			For each $z \in \Om_{\ve}$, $\zeta \in \Om_{\ve} \sm \ov{\Om}, 0 \leq i, j < \infty$, the following holds:
			\eq{West}
			| D_{z}^{i} D_{\zeta}^{j} W(z, \zeta) | \leq C_{i,j}(\Om, \|\rho_{0}\|_{C^{k+2}(\ov{\Om})}, \del_0) (1+  \del(\zeta)^{k+1 -j}). 
			\eeq
\end{enumerate} 
	\end{prop} 
\begin{proof}
  See \cite[Proposition 5.1]{Gong19} for the case $k=0$. The general case follows by similar construction and we omit the details. 
\end{proof} 

When the boundary is $C^{\infty}$, we have an analogous version of \rp{Prop::Reg_H-R}. In this case the usual Henkin-Ramirez function $W$ suffices. Namely, $W$ satisfies above properties \ref{Item::Reg_H-R::1} and \ref{Item::Reg_H-R::2}, and estimate \re{West} is replaced by 
\begin{equation} \label{West2}
  | D^i_z D^j_\zeta W (z, \zeta)| \leq C_{ij}, \quad 0 \leq i, j < \infty.   
\end{equation}

	We call $\Phi(z, \zeta)$ a holomorphic support function. 
	Near every $\zeta^{\ast} \in b\Om $, one can find a neighborhood $\mc{V}$ of $\zeta^{\ast}$ such that for all $z \in \V$, there exists a coordinate map $\phi_{z}: \mc{V} \to \C^n$ given by $\phi_{z}: \zeta \in \V \to (s, t) = (s_{1}, s_{2}, t_{3},\dots, t_{2n})$, where $s_1 = \rho(\zeta)$ (thus $s_1(\zeta) \approx \del(\zeta)$ for $\zeta\in\V\sm\Omega$). 
	Moreover for $z \in \V \cap \Om$, $\zeta \in \V \sm\overline \Om$, $\Phi  = W(z,\zeta) \cdot (\zeta -z)$ satisfies 
	\eq{Phie1}
	|\Phi (z, \zeta) | \geq c \left( \del(z) + s_{1} + |s_{2}| + |t|^{2} \right),  \quad \del(z) := \dist(z, b\Om), 
	\eeq
	\eq{Phie2}
	|\Phi(z, \zeta)| \geq c |z - \zeta|^{2}, \quad \quad |\zeta -z| \geq c( s_1 + |s_2| + |t|), 
	\eeq
	for some constant $c$ depending on the domain.  
	This is a consequence of \rp{Prop::Reg_H-R} \ref{Item::Reg_H-R::2}. The reader can also refer to \cite{Gong19} for details.   
	
	From now on, we shall fix an open set $\U=\U_{\Om}$ such that $ \Om \subset \subset \U \subset \subset \Om_{\ve} $, and we will use our extension operator $\Ec$ given in \rp{ext_op} with $\supp \Ec \var \subset \U$ for all $\var$. Our homotopy operator $\Hc_q $ has the form 
  \begin{equation} \label{Hqsum}
    	\Hc_q \var =  \Hc_q^0 \var + \Hc_q^1 \var,   
  \end{equation} 
where
	\begin{align} \label{H0H1} 
	\Hc_q^0 \var(z) := \int_{\U} K^0_{0, q-1}(z,\cdot) \we \Ec \var,  
	\quad  
	\Hc_q^1 \var(z) :=  \sum_{|\gm| \leq l}(-1)^{|\gm|} \int_{\U \sm \ov\Om} D^{\gm}(K^{01}_{0,q}(z,\cdot)) \we \Sc^{l,  \gm}_{\Om} [\dbar, \Ec ]\var, 
	\end{align} 
	Here $l$ can be chosen to be any integer greater or equal to $k+1$ (see \rp{Prop::Hq_Ckbdy}). The kernels $K^0$ and $K^{01}$ are given by 
	\begin{equation} \label{K0} 
	K^{0} (z, \zeta) = \frac{1}{(2 \pi i)^{n}} \frac{\left<\ov{\zeta} - \ov{z} \, , \, d \zeta \right>}{|\zeta -z |^{2}} \we \left( \dbar_{\zeta,z} \frac{ \left< \ov{\zeta} - \ov{z} \, , \, d \zeta \right>}{|\zeta -z|^{2}} \right)^{n-1}, \quad \dbar_{\zeta,z} = \dbar_{\zeta} + \dbar_{z} ; 
	\end{equation} 
	\begin{align} \label{K01} 
	& K^{0, 1} (z, \zeta) 
	= \frac{1}{(2 \pi i)^{n}} \frac{\left<\ov{\zeta} - \ov{z} \, , \, d \zeta \right>}{|\zeta -z |^{2}} \we \frac{\left< W(z,\zeta), d \zeta  \right>}{\left< W(z,\zeta) \, , \, \zeta -z \right>} 
	\\ \nonumber & \qquad \we \sum_{i+j=n-2} \left[ \frac{\left< d\ov{\zeta} - d\ov{z} \, , \, d \zeta \right>}{|\zeta -z |^{2}} \right]^{i}   \we \left[ \dbar_{\zeta,z} \frac{ \left< W(z,\zeta)  ,  d \zeta  \right>}{\left< W(z,\zeta), \zeta -z \right>} \right]^{j}.
	\end{align}
	We set $K^{1}_{0,-1} = 0$ and $K^{0,1}_{0,-1} =0$.

	\begin{prop} \label{Prop::H0_est} 
		Let $s \in \R$ and $1 < p < \infty$. 
		Suppose $\var \in H^{s,p}_{(0, q)} (\Om)$. Then $\Hc^0_q \var \in H^{s+1,p}_{(0, q-1)} (\C^n)$. More precisely, there exists a constant $C = C(\Om,s, p)$ such that 
		\begin{equation}  \label{H0est} 
		\| \Hc^0_q \var  \|_{H^{s+1,p} (\C^n)} \leq C \| \var \|_{H^{s,p} (\Om)}. 
		\end{equation} 
	\end{prop}
	\begin{proof}
	The proof for positive integer $s$ can be found in \cite[Proposition 3.2]{Shi21}. We now prove the case $s \leq -1$. The rest is done by interpolation. 
	
	 $\Hc^0_q \var$ can be written as a linear combination of
		\[
		\int_{\U} \frac{\ov{\zeta^i - z^i}}{|\zeta- z|^{2n}} f(\zeta) \, dV(\zeta)  = \pa_{z_i} \Nc \ast f(z), 
		\]
		where $f$ is a coefficient function of $\Ec \var$ and $\Nc (x):= c_n \frac{1}{|x|^{2n-2}}$ is the Newtonian potential. In particular $f$ has compact support in $\U$. Let $g \in H^{-s -1,p'}_0 (\U)$, where $-s -1 \geq0$ ($s \leq -1$) and $\frac1p + \frac{1}{p'} = 1$. Then 
		\begin{align*}
		\left< D \Nc \ast f, g \right> 
		&= - \left< f, D \Nc \ast g \right>. 
		\end{align*} 
		Now $D \Nc \ast g$ is in $H^{-s,p'}_{\loc}(\C^n)$ by the result for positive index. Since $f \in H^{s,p}_c (\U)$ has compact support, we have 
		\[  
		| \left< f, D \Nc \ast g  \right> | 
		\leq \| f \|_{H^{ s,p} (\U)} \| D \Nc \ast g \|_{H^{-s,p' } (\U)} 
		\leq \| f \|_{H^{s,p}(\U)} \| g \|_{H^{-s-1,p'} (\U)}. 
		\]    
		Hence $|  \left< D \Nc \ast f, g \right> | \leq \| f \|_{H^{s,p} (\U)} \| g \|_{H^{-s-1,p'} (\U)}$. Since $H^{s+1,p} (\U) =  \left( H^{-s-1, p'}_0 (\U) \right)'$ (see \rp{Prop::DualSpace} \ref{Item::DualSpace::H0Char}), we see that $D \Nc \ast f \in  H^{s+1,p}(\U)$ and \re{H0est} holds. 
	\end{proof}
	
	In order to estimate $\Hc_q^1$, we need the following lemma. 
	
	\begin{lemma}\label{lem::intest} 
		Let $\beta \geq 0$, $\all > -1$, $\all < \beta - \yh $, and let $0 < \del < \yh $. 
		Define $\del(z):= \dist(z, b \Om)$. Then for any $z, \zeta \in \Uc$: 
		\[
		\int_{\Uc \sm \Om} \frac{ \del(\zeta)^{\all}  \, dV (\zeta)}{|\Phi(z, \zeta)|^{2+\beta} |\zeta -z|^{2n-3} } \leq C \del(z)^{\all - \beta + \yh}; \quad 
		\int_{\Om} \frac{ \del(z)^{\all} \, dV(z)}{|\Phi (z, \zeta)|^{2+\beta} |\zeta -z|^{2n-3} } \leq C \del(\zeta)^{\all - \beta + \yh}.  
		\] 
	where the constants depends only on $\alpha,\beta$, $\Om$ and $\Uc$. 
	\end{lemma}
	\begin{proof}
		By the remark after \rp{Prop::Reg_H-R}, near each $\zeta_0 \in b \Om$ there exists a neighborhood $\Vc_{\zeta_0}$ of $\zeta_0$ and a coordinate system $\zeta \mapsto ( s =(s_1, s_2), t ) \in \R^2 \times \R^{2n-2}$ such that
		\begin{align} \label{zeta_coord_est} 
		|\Phi(z, \zeta)| \gtrsim \del(z) + |s_1| + |s_2| + |t|^2,  \quad
		|\zeta - z| \gtrsim |(s_1, s_2, t)|.  
		\end{align}
		for $z \in \Vc_{\zeta_0} \cap \Om$ and $\zeta \in \Vc_{\zeta_0} \sm \Om$, with $|s_1 (\zeta) | \approx \del(\zeta)$. By switching the roles of $z$ and $\zeta$, we can also find a coordinate system $z \mapsto (\ti{s}= (\ti{s}_1, \ti{s}_2), \ti{t}) \in \R^2 \times \R^{2n-2}$ such that
		\begin{align} \label{z_coord_est} 
		|\Phi(z, \zeta)| \gtrsim \del(\zeta) + |\ti{s}_1| + |\ti{s}_2| + |\ti{t}|^2,  \quad
		|\zeta - z| \gtrsim |(\ti{s}_1, \ti{s}_2, t)| 
		\end{align} 
		for $z \in \Vc_{\zeta_0} \cap \Om$ and $\zeta \in \Vc_{\zeta_0} \sm \Om$, with $|\ti{s}_1 (z) | \approx \del(z)$. 
		By partition of unity in both $z$ and $\zeta $ variables and \re{zeta_coord_est}, we have up to a constant depending only $\Om$,  
		\begin{align*}
		\int_{\Uc \sm \Om} \frac{ \del(\zeta)^{\all}  \, dV (\zeta)}{|\Phi(z, \zeta)|^{2+\beta} |\zeta -z|^{2n-3} }
		&\lesssim \int_0^1 \int_0^1 \int_0^1 \frac{s_1^{\all} t^{2n-3} \, ds_1 \, ds_2 \, dt}{(\del(\zeta) + s_1 + s_2 + t^2) ^{2+\beta}  (\del + s_1+ s_2 + t)^{2n-3} }.  
		\end{align*} 
		By \cite[Lemma 6.5]{S-Y21-1}, the integral is bounded by a constant multiple of $\del(z)^{\all-\beta + \yh}$, which proves the first inequality. The second inequality follows by the same way. 
	\end{proof}

	\begin{prop} \label{Prop::Hq_Ckbdy}  Let $k$ be a positive integer and let $\Om\subset \C^n$ be a bounded strictly pseudoconvex domain with $C^{k+2}$ boundary. 
	For $q \geq 1$, let $\Hc_q^1 \var$ be given by \re{Hq_def}, where we choose $l$ to be any positive number greater or equal to $k+1$. Let $  s > -k + \frac{1}{p}$ and $1 < p < \infty$. Then for any non-negative integer $m>s+ k + \frac12$, there exists a constant $C = C(\Om,k,m,s,p)$ such that for all $\var\in H^{s,p}_{(0,q)}(\Omega)$,
		\[
		\| \del^{m - s-\yh} D^m \Hc_q^1 \var \|_{L^p(\Om)} \leq C \| \var \|_{H^{s,p} (\Om)}. 
		\]
	\end{prop}
	\begin{proof} 
		We estimate the integral 
		\begin{equation} \label{main_int}  
		\int_{\Om} \del(z)^{(m - s-\yh)p} \left| \sum_{|\gm| \leq l}\int_{\U \sm \ov\Om} D_z^m D^{ \gm}_{\zeta} K(z, \cdot) \we \Sc^{l, \gm}([\dbar, \Ec] \var) \right|^p \, dV(z). 
		\end{equation}
		The idea is to choose $l$ large enough so that $\Sc^{l, \gm} [\dbar, \Ec] \var$ lies in some positive index space. Let $f$ be a coefficient function of $[\dbar, \Ec]\var$ so that $f \in H^{s-1,p}(\Uc \sm \ov{\Om})$. By $\Sc^{l, \gm}f $ we mean each component of $\Sc^{l, \gm} ([\dbar, \Ec]\var) $ . In the computation below we shall fix a multi-index $\gm$ and write $\Sc^{l}$ without causing ambiguity. 
		By Lemma \ref{comm_antider}, $\Sc^{l} f \in H^{s+l-1,p}(\Uc \sm \ov{\Om})$. Since $l \geq k+1$, and $s > \frac{1}{p} - k$, we have $ l > -s +1$, or $s + l -1 > 0$. Since $f \equiv 0 $ in $\Om$, by \rp{Prop::DualSpace} \ref{H_0space} we have $\Sc^{l} f \in H^{s +l -1,p}_0 (\U \sm \ov{\Om})$. 
		
We now estimate the inner integral in \re{main_int} which we shall denote by $\Kc f$. In view of \re{K01}, we can  write $D^m_z D^\gm_\zeta K$ as a linear combination of 
\[
		\int_{\U \sm \ov{\Om}} \Sc^l f(\zeta) \frac{P_{\mu_0, \mu_1} (\zeta,z) }{\Phi^{n-i+\mu_2} (z, \zeta) |\zeta -z|^{2i+\mu_3-1}} \, dV(\zeta), \quad \sum_{j=0}^3 \mu_j \leq m+l, \quad 1 \leq i \leq n-1, 
\]  
where $P_{\mu_0, \mu_1}$ takes the form   
\begin{equation} \label{P_mu0_mu1}  
\begin{gathered}
   \prod_{i=1}^n (\zeta_i - z_i)^{\thh_i} (\ov{\zeta_i - z_i} )^{\thh_i'} (D_z^{\si_1+1} D_\zeta^{\nu_1+1} W) \cdots  (D_z^{\si_L+1}   
D_\zeta^{\nu_L+1} W) (D_{z,\zeta} W) ;
 \\ 
 \sum_{i=1}^M \si_i \leq \mu_0 \leq m, \quad \sum_{j=1}^L \nu_j \leq \mu_1 \leq l, \quad \thh_i, \thh_i' \leq 
m+l,
\end{gathered}
\end{equation}
and $D_{z,\zeta} W $ denotes a product of first derivatives of $W$ in $z$ and $\zeta$.  
Since $ |\zeta -z|^2 \lesssim | \Phi (z, \zeta)| \lesssim |\zeta -z|$, it suffices to consider the case when $\mu_3= 0$ and $i = n-1$, namely, 
		\begin{equation} \label{int_der } 
		\int_{\U \sm \ov{\Om}} \Sc^l f(\zeta) \frac{P_{\mu_0,\mu_1} (\zeta,z) }{\Phi^{\mu_2+1} (z, \zeta) |\zeta -z|^{2n-3}} \, dV(\zeta), \quad \mu_0 + \mu_1 + \mu_2 \leq m + l. 
		\end{equation} 
By \re{West}, we have $|D_z^i  
D_{\zeta}^{j+1} W(z, \zeta) | \lesssim 1 + \del(\zeta)^{k-j}$, for $\zeta \in \U \sm \ov{\Om}$ and $k \geq 1$. We make the following observation: if we have already taken $k (<l)$ $\zeta$ derivatives on $W$, then the integral becomes worse if we distribute one more $\zeta$ derivative to $W$ rather than $\Phi$. On the other hand if we take less than $k$ $\zeta$-derivatives on $W$, then the integral becomes worse if we move all the $ \zeta$ derivatives to $\Phi$ instead. Therefore we only have to estimate \re{int_der } for $(\mu_0 = \mu_1 = 0, \mu_2 = m + l)$ and for $ (\mu_0 = 0, 
\mu_1 = l, \mu_2 = m) $, which we denote by $\mc{K}_0 $ and $\mc{K}_1$: 
		\begin{gather*} 
		\mc{K}_0 f (z) := \int_{\U \sm \ov{ \Om}} \Sc^l f (\zeta) \frac{P_{0,0}(\zeta,z)} 
{\Phi^{m+l+1}(z, \zeta)  |\zeta-z|^{2n-3}} \, dV(\zeta), 
		\\ 
		\mc{K}_1f (z) := \int_{\U \sm \ov{ \Om}} \Sc^l f (\zeta) \frac{P_{0,l}(\zeta,z) }{\Phi^{m+1}(z, \zeta)  |\zeta-z|^{2n-3}} \, dV(\zeta). 
		\end{gather*}  
		We will show that for any $s \in \R$, 
		\begin{align} \label{K0est}
		\int_{\Om} \del(z) ^{(m-s - \yh)p } |\Kc_0  f (z) |^p  \, dV(z)  
		\lesssim \|\Sc^l f \|_{H^{s-1+l, p} (\C^n)}^p, 
		\end{align} 
		if $m$ and $l$ are chosen to be sufficiently large, and also for any $ s > -k + \frac{1}{p}$, 
		\begin{align} \label{K1est} 
		\int_{\Om} \del(z)^{(m-s -k -\yh) p} |\Kc_1f (z)|^p \, dV(z)
		\leq \| \Sc^l f \|_{H^{s-1+l,p} (\C^n)}^p, 
		\end{align}  
		if $m$ is sufficiently large.  
		To simplify the notation we shall denote the kernel of $\Kc_0$ and $\Kc_1$ by $B_0$ and $B_1$ respectively. 
	By \re{West} and \re{P_mu0_mu1}, we see that 
 $|P_{0,0}(\zeta,z)| \lesssim 1$ and $|P_{0,l}(\zeta,z)| \lesssim \del(\zeta)^{k-l} $, for $k < l$ and $\del(\zeta) < 1$ (which is true for $\zeta \in \Uc \sm \ov \Om$).   
		Thus $B_0$ and $B_1$ satisfy 
		\begin{gather} \label{B0_est}  
		| B_0 (z,\zeta) | \lesssim \frac{1 }{|\Phi (z, \zeta)|^{m +l+1}  |\zeta -z|^{2n-3}};  
		\\ \label{B1_est} 
		| B_1 (z,\zeta) | \lesssim  \frac{\del(\zeta)^{k-l} }{|\Phi (z, \zeta)|^{m +1}  |\zeta -z|^{2n-3}}. 
		\end{gather}

	\medskip \noindent
	\tit{Case 1: Estimate of the $\Kc_0$ Integral.} 
	
		First we estimate \re{main_int} where the inner integral is $\Kc_0$. 
		Let $\eta$ be some number to be specified; we have
		\begin{align*}
		| \Kc_0 f (z) |  &\leq \int_{\U \sm \ov{\Om}} |  B_0 (z, \zeta) |^{\frac{1}{p}} 
		|  B_0 (z, \zeta) |^{\frac{1}{p'}} | \Sc^l f(\zeta) | \, dV(\zeta)
		\\ &=  \int_{\U \sm \ov{\Om}} \del(\zeta) ^{-\eta} | \del(\zeta)^{\eta} B_0 (z, \zeta) |^{\frac{1}{p}} 
		| \del(\zeta)^{\eta} B_0 (z, \zeta) | ^{\frac{1}{p'}} | \Sc^l f(\zeta) | \, dV(\zeta). 
		\end{align*}
		Applying H\"older's inequality we get  
		\begin{equation}  \label{K0_Holder}
		| \Kc_0 f (z) |^p \leq 
		\left[ \int_{\U \sm \ov{\Om}}  \del(\zeta)^{- \eta p +\eta} | B_0 (z, \zeta) | | \Sc^l f(\zeta)|^p \, dV (\zeta) \right] 
		\left[ \int_{\U \sm \ov{\Om}} \del(\zeta)^{\eta} | B_0 (z, \zeta) | \, dV(\zeta)  \right]^{\frac{p}{p'}}. 
		\end{equation} 
		By \re{B0_est} and \rl{lem::intest}, 
			\begin{align} \label{B0int}
		\int_{\U \sm\ov{\Om}} \del(\zeta)^{\eta} | B_0 (z, \zeta) | \, dV(\zeta) 
		&\leq \int_{ } \frac{\delta(\zeta)^\eta\, dV(\zeta) }{|\Phi (z, \zeta)|^{m +l+1}  |\zeta -z|^{2n-3}} 
		\\ \nn&\lesssim \del(z)^{\eta - (m+l-1) + \yh} 
		= \del(z)^{\eta - (m+l) +  \frac{3}{2}}, 
		\end{align} 
		where we applied \rl{lem::intest} using $\all = \eta$, $\beta = m+l-1$ and by choosing 
		\begin{equation}  \label{K0eta1}
		-1 < \eta < \beta - \yh =  m + l  - \frac{3}{2}. 
		\end{equation}
		Here the constant in the estimate depends only on the domain $\Om$ and the defining function $\rho$, and in particular is independent of $\zeta$. Using \re{K0_Holder} and \re{B0int} and applying Fubini's theorem gives 
		\begin{align} \label{K0_Fubini}  
		\int_{\Om} \del(z) ^{(m-s - \yh)p } |\Kc_0  f (z) |^p  \, dV(z)
		&\lesssim \int_{\Om} \del(z)^{\tau} \left( \int_{\U \sm \ov{\Om}} \del (\zeta)^{(1-p) \eta}  |B_0 (z, \zeta)| |\Sc^lf (\zeta)|^p \, dV(\zeta) \right) \, dV(z)
		\\ \nn &\lesssim  \int_{\U \sm \ov{\Om}}  \del(\zeta)^{(1-p) \eta} \left[ \int_{\Om} \del(z)^{\tau} | B_0 (z, \zeta) | \, dV (z) \right] |\Sc^l f(\zeta) |^p \, dV(\zeta),   
		\end{align}
		where 
		\begin{align} \label{K0tau}   
		\tau &= p \left( m-s-\yh \right) + \left( \eta - m - l + \frac{3}{2} \right) \frac{p}{p'}  &= -sp + \eta (p-1) + (1-l)p +  \left( m + l - \frac{3}{2}  \right). 
		\end{align}
		By estimate \re{B0_est}, the inner integral satisfies 
	\begin{equation} \label{B0int2}
		\int_{\Om} \del(z)^{\tau} |B_0(z, \zeta)| \, dV(z)
		\lesssim \int_{\Om} \frac{\del(z)^{\tau}}{|\Phi(z, \zeta)|^{m+l+1}|\zeta-z|^{2n-3}} \, dV(z).
\end{equation}
		To estimate this integral we apply \rl{lem::intest} with $\all =  \tau$ and $\beta = m + l -1$, and we need 
		$-1 <  \tau < \beta - \yh = m + l - \frac{3}{2}$; thus in view of \re{K0tau}, $\eta$ needs to satisfy
		\begin{equation} \label{K0eta2} 
		\frac{1}{p-1}  \left[ (s+ l -1) p - m -l + \yh\right]
		<  \eta < \frac{1}{p-1} (s+l-1) p.
		\end{equation} 
		where the first inequality follows from $\tau >-1$ and the second inequality follows from $\tau < m+l - \frac{3}{2}$. 
		Now, in view of \re{K0eta1} and \re{K0eta2}, we need to choose  
		\begin{equation}  \label{K0eta}
		\max \left\{ -1,    \frac{1}{p-1}  \left[ (s+ l -1) p - m -l + \yh\right] \right\}    
		< \eta < \min \left\{ m+l- \frac{3}{2}, \frac{1}{p-1} (s+l-1) p \right\}.
		\end{equation}
		An easy computation shows that the above admissible range for $\eta$ is non-empty if and only if $s$ satisfies 
		\begin{equation} \label{K0_s_range} 
		\frac{1}{p} -l  < s< m - \yh + \frac{1}{p}. 
		\end{equation}
		Now with $\eta$ chosen in the range \re{K0eta}, we can apply \rl{lem::intest} to the integral \re{B0int2} to get
		\begin{align*}
		\int_{\Om} \del(z)^{\tau} | B_0 (z, \zeta) | \, dV(z) 
		\lesssim \del(\zeta)^{\tau- (m+l-1) + \yh} = \del(\zeta)^{-sp + \eta(p-1) + (1-l) p }, 
		\end{align*} 
		where the constant is independent of $\zeta$.  Substituting the estimate into \re{K0_Fubini} we obtain 
		\begin{align*}
		\int_{\Om} \del(z) ^{(m-s - \yh)p } |\Kc_0  f (z) |^p  \, dV(z) 
		&\lesssim  \int_{\U \sm \ov{\Om}} \del(\zeta)^{(1-p) \eta} \del(\zeta)^{-sp + \eta(p-1) + (1-l) p } | \Sc^l f(\zeta) |^p \, dV(\zeta) 
		\\ &\lesssim \int_{\U \sm \ov{\Om}} \del(\zeta)^{(-s + 1 -l) p } | \Sc^l f (\zeta) |^p \, dV(\zeta) 
		\\ &\lesssim \|\Sc^l f \|_{H^{s-1+l, p} (\C^n)}^p ,  
		\end{align*}
		where in the last step we used \rp{Prop::C3} and $s-1+l \geq 0$.

		Note that in view of \re{K0_s_range}, we can choose $l$ and $m$ arbitrarily large so that the above estimate holds for any $s \in \R$. 
		
	\medskip 
\nid \tit{Case 2: Estimate of $\Kc_1$ integral.} 
	
		Next we prove estimate \re{K1est}. We shall show that   
		\[ 
		\int_{\Om} \del(z)^{(m-s -j -\yh) p} |\Kc_1f (z)|^p \, dV(z)
		\leq \| \Sc^l f \|_{H^{s-1+l,p} (\C^n)} ^p
		\]
		holds for any $j$ with $j \leq k$ and $s> -j + \frac{1}{p}$. By setting $j=k$ we get \re{K1est}.  
		By the same computation as in \re{K0_Holder} we see that
		\begin{align} \label{K1_Holder}  
		| \Kc_1 f (z) |^p \leq 
		\left[ \int_{\U \sm \ov{\Om}}  \del(\zeta)^{- \eta p +\eta} | B_1 (z, \zeta) | | \Sc^l f(\zeta)|^p \, dV (\zeta) \right] 
		\left[ \int_{\U \sm \ov{\Om}} \del(\zeta)^{\eta} | B_1 (z, \zeta) | \, dV(\zeta)  \right]^{\frac{p}{p'}}. 
		\end{align} 
		where $\eta$ is to be chosen. Using \re{B1_est} we have
		\begin{align*}   
		\int_{\U \sm \ov{\Om}} \del(\zeta)^{\eta} |B_1(z, \zeta)| \, dV(\zeta)
		&\leq \int_{\U \sm \ov{\Om}}  \frac{\del(\zeta)^{\eta+k-l} }{|\Phi (z, \zeta)|^{m+1} |\zeta-z|^{2n-3}} \, dV(\zeta). 
		\end{align*} 
		We now apply \rl{lem::intest} with $\all = \eta + k -l$ and $\beta = m-1$ to get 
		\begin{equation} \label{B1_zeta_int} 
		\int_{\U \sm \ov{\Om}}  \del(\zeta)^{\eta} |B_1(z, \zeta)| \, dV(\zeta) 
		\lesssim \del(z)^{\eta+k-l - (m-1) + \yh} =\del(z)^{\eta+k -l-m + \frac{3}{2}}. 
		\end{equation} 
		Here we need to choose  $-1 < \eta+ k-l < \beta - \yh = m -\frac{3}{2}$ or  
		\begin{equation} \label{K1eta1}
		l - k - 1 < \eta < l - k + m - \frac{3}{2}. 
		\end{equation}
		Applying Fubini's theorem and using \re{K1_Holder} and \re{B1_zeta_int}, we have 
		\begin{align} \label{K1_Fubini}  
		\int_{\Om} \del(z)^{(m-s-j-\yh)p } |\Kc_1 f(z) |^p  \, dV(z) 
		&\lesssim  \int_{\Om} \del(z)^{\kappa} \left( \int_{\U \sm \ov{\Om}} \del (\zeta)^{(1-p) \eta}  |B_1 (z, \zeta)| | \Sc^l f (\zeta)|^p \, dV(\zeta) \right) \, dV(z)
		\\ \nn &= \int_{U \sm \ov{\Om}}  \del(\zeta)^{(1-p) \eta} \left[ \int_{\Om} \del(z)^{\kappa} | B_1 (z, \zeta) | \, dV (z) \right] | \Sc^l f(\zeta) |^p \, dV(\zeta), 
		\end{align}  
		where 
		\begin{equation} \label{K1kappa}  
		\begin{aligned}
		\kappa &= \left( m-s-j- \yh \right)p + \frac{p}{p'} \left(  \eta + k - l - m + \frac{3}{2} \right)
		\\  &=  (-s-l+1) p + \eta(p-1) + l-k + p(k-j) + m - \frac{3}{2} . 
		\end{aligned} 
		\end{equation}
		By estimate \re{B1_est}, we have 
		\begin{align} \label{B1_z_int}  
		\int_{\Om} \del(z)^{\kappa} | B_1 (z, \zeta) | \, dV (z) 
		&= \del(\zeta)^{ k-l} \int_{\Om} \frac{\del(z)^{\kappa} }{|\Phi(z, \zeta)|^{m+1} |\zeta-z|^{2n-3}} \, dV(z). 
		\end{align}
		To estimate this integral we apply \rl{lem::intest} with $\all = \kappa$ and $\beta = m-1$, which requires $-1 < \kappa < \beta -  \yh = m - \frac{3}{2} $. In view of \re{K1kappa}, $\eta$ needs to satisfy
		\begin{equation} \label{K1eta2}
		p' (s+l-1 + j-k) + \frac{1}{p-1} (k-l-m + \yh) < \eta < p'( s+ l -1+ j-k) + \frac{k-l}{p-1}. 
		\end{equation}
		Now for the existence of $\eta$ satisfying both \re{K1eta1} and \re{K1eta2}, $s$ needs to satisfy
		\begin{align} \label{K1_s_range} 
		- j  + \frac{1}{p} < s < -j + \frac{1}{p} + m - \yh. 
		\end{align} 
		In particular for any $s > -j + \frac{1}{p}$, we can choose $m$ large enough so that \re{K1_s_range} holds. 
		With $\eta$ chosen to be in the range \re{K1eta1} and \re{K1eta2}, we can then apply \rl{lem::intest} to integral \re{B1_z_int} to get     
		\[
		\int_{\Om} \frac{\del(z)^{\kappa} }{|\Phi(z, \zeta)|^{m+1} |\zeta-z|^{2n-3}} \, dV(z)  
		\lesssim \del(\zeta)^{\kappa - (m-1) + \yh} 
		= \del(\zeta)^{(-s-l+1)p + \eta(p-1) + l-k + p(k-j)}, 
		\]
		where the constant is independent of $\zeta$.
		Therefore from \re{B1_z_int} we get
		\begin{align} \label{B1_zint_est}
		\int_{\Om} \del(z)^{\kappa} | B_1 (z, \zeta) | \, dV (z) 
		\lesssim \del(\zeta)^{(-s-l+1)p + \eta(p-1) + p(k-j)}. 
		\end{align}
		
		Now by using \re{K1_Fubini} and \re{B1_zint_est} we get 
		\begin{align*}
		\int_{\Om} \del(z)^{(m-s-j-\yh)p } |\Kc_1 f(z) |^p  \, dV(z)  
		&\lesssim \int_{\U \sm \ov{\Om}} \del(\zeta)^{(1-p) \eta + (-s-l+1)p + \eta(p-1)  + p(k-j)} | \Sc^l f(\zeta) |^p \, dV(\zeta) 
		\\ &\lesssim \int_{\U \sm \ov{\Om}} \del(\zeta)^{(-s -l+1) p + p(k-j)}  | \Sc^l f(\zeta) |^p \, dV(\zeta) 
		\\ &\lesssim \| \Sc^l f \|_{H^{s-1+l,p} (\C^n)}^p , 
		\end{align*}
		where in the last inequality we use $k \geq j$ and \rp{Prop::C3} with $s-1+l \geq 0$. This proves \re{K1est} and concludes the proof.  
	\end{proof} 
	
	Next we extend the weighted estimate of \rp{Prop::Hq_Ckbdy} to all lower-order derivatives. 
	\pr{lord_est}
	Keeping the assumptions of \rp{Prop::Hq_Ckbdy}, one has the following
	\[
	\| \del(z)^{m-s-\yh} D^{r} \Hc^1_q \var \|_{L^p (\Om)} \leq C (\Om, p) \| \var \|_{H^{s,p}(\Om)}, 
	\quad 0 \leq r \leq m. 
	\]
	\epr 
	\begin{proof} 
	Following the proof of \rp{Prop::Hq_Ckbdy}, it is sufficient to show that for any $s \in \R$,
	\begin{align} \label{K0'est}
	\int_{\Om} \del(z) ^{(m-s - \yh)p } |\Kc'_0  f (z) |^p  \, dV(z)  
	\lesssim \|\Sc^l f \|_{H^{s-1+l, p} (\C^n)}^p, \quad f = [\dbar, E] \var
	\end{align} 
	if $m$ and $l$ are chosen to be sufficiently large. And also for any $ s > -k + \frac{1}{p}$, 
	\begin{align} \label{K1'est} 
	\int_{\Om} \del(z)^{(m-s -k -\yh) p} |\Kc'_1 f (z)|^p \, dV(z)
	\leq \| \Sc^l f \|_{H^{s-1+l,p} (\C^n)}, 
	\end{align}  
	if $m$ is sufficiently large. Here  
	
	\begin{gather} 
	\mc{K}'_0 f (z) := \int_{\U \sm \ov{ \Om}} \Sc^l f (\zeta) \frac{P_{0,0}(\zeta,z)}{\Phi^{r +l+1}(z, \zeta)  |\zeta-z|^{2n-3}} \, dV(\zeta), \quad 1 \leq r \leq m;  
	\\ 
	\mc{K}'_1f (z) := \int_{\U \sm \ov{ \Om}} \Sc^l f (\zeta) \frac{P_{0,l}(\zeta, z)}{\Phi^{r+1}(z, \zeta)  |\zeta-z|^{2n-3}} \, dV(\zeta), \quad 1 \leq r \leq m,  
	\end{gather}  
with $P_{0,0}, P_{0,l}$ given by formula \re{P_mu0_mu1}.  
	Now $\mc{K}'_0 f$ and $\mc{K}'_1 f$ are both $C^{\infty}$ for $z$ in the interior of $\Om$. By a partition of unity we can assume that $z \in \V \cap \Om$ and $\zeta \in \V \sm \ov{\Om}$, where $\V$ is the neighborhood in the remark right after \rp{Prop::Reg_H-R}. Since $|\Phi(z, \zeta)| \leq C|z-\zeta|$, we can assume that $|\Phi(z,\zeta)| < 1$ for any $z, \zeta \in V$. Therefore for any $r \leq m$, 
	\begin{align} \label{K0'est2}  
	|\Kc_0' f(z)| \leq 
	\int_{\U \sm \ov{ \Om}}\frac{ |\Sc^l f (\zeta)|  |P_{0,0}(\zeta,z)|}{|\Phi^{r +l+1}(z, \zeta)| |\zeta-z|^{2n-3}} \, dV(\zeta) 
	\leq  \int_{\U \sm \ov{ \Om}} \frac{  |\Sc^l f (\zeta)| |P_{0,0}(\zeta,z)|}{|\Phi^{m+l+1}(z, \zeta)| |\zeta-z|^{2n-3}} \, dV(\zeta) 
	\end{align}
	and 
	\begin{align}   \label{K1'est2}  
	|\Kc_1' f(z)| \leq 
	\int_{\U \sm \ov{ \Om}}\frac{ |\Sc^l f (\zeta)|  |P_{0,l}(\zeta,z)|}{|\Phi^{r+1}(z, \zeta)| |\zeta-z|^{2n-3}} \, dV(\zeta) 
	\leq  \int_{\U \sm \ov{ \Om}} \frac{  |\Sc^l f (\zeta)| |P_{0,l}(\zeta,z)|}{|\Phi^{m+1}(z, \zeta)| |\zeta-z|^{2n-3}} \, dV(\zeta) . 
	\end{align}
Inequalities \re{K0'est} and \re{K1'est} then follow by the proof of \rp{Prop::Hq_Ckbdy}, if $s > -k+\frac{1}{p}$, and $m, l$ are sufficiently large. 
	\end{proof} 

	We can now prove \rt{Thm::EstHq} using \rp{lord_est} and \rp{H-L}. 
	
	\begin{proof}[Proof of \rt{Thm::EstHq}]
	We have for any $s > -k + \frac{1}{p}$ and $m \in \Z^+$ sufficiently large, 
	\begin{align} \label{H1_last_est}
	\| \Hc_q^1 \var  \|_{H^{s+ \yh,p} (\Om)} 
	\lesssim \sum_{|\gm| \leq m} \| \del^{m- s - \yh} D^{\gm} \Hc^1_q \var \|_{L^p(\Om)}
	\lesssim  \| \var \|_{H^{s,p}(\Om)}. 
	\end{align} 
	Together with \rp{Prop::H0_est} for the operator $\Hc^0_q$, we have proved \rt{Thm::EstHq}.  
	\end{proof}

We now state and prove the version of \rt{Thm::EstHq} for $C^\infty$ boundary. The homotopy operator in this case still takes the form \re{Hqsum}-\re{H0H1}, but now for the kernel $K^{01}$ given in \re{K01} we use the classical Henkin-Ramirez function $W$ (see the remark after \rp{Prop::Reg_H-R}), and also in the formula for $\Hc^1_q$ (see \re{H0H1}) we take $l \geq -s +1$.   By Remark \ref{Rem::l_ind}, the definition of $\Hc_q$ does not depend the choice of $l$. 

\begin{thm}\label{Thm2::EstHq} 
	Let $\Omega\subset\C^n$ be a bounded strictly pseudoconvex domain with $C^{\infty}$ boundary. Let $1\le q\le n$ and let $\Hc_q$ be defined by \eqref{Hqsum}-\eqref{K01}, where $l$ is any positive integer greater or equal to $-s+1$. Then for any $s \in \R$ and $1<p<\infty$, $\Hc_q$ defines a bounded linear operator $\Hc_q: H_{(0,q)}^{s,p} (\Om)\to H^{s+\frac12,p}_{ (0,q-1)}(\Om)$.
\end{thm} 
\begin{proof}
  The proof follows the same strategy as for \rt{Thm::EstHq}, namely proving \rp{Prop::Hq_Ckbdy}, \rp{lord_est} and then using estimate \re{H1_last_est}. The only difference here is that since $W$ satisfies the estimate \re{West2}, in the proof of \rp{Prop::Hq_Ckbdy} we only have to estimate the integral 
  \[
    \int_\Om \del(z)^{(m-s-\yh)p} |\Kc_0 f (z)|^p \, dV(z) 
    \lesssim \| \Sc^l [\dbar, \Ec]\var \|^p_{H^{s-1+l,p}(\Om)}
    \lesssim \| \var \|_{H^{s,p}(\Om)}^p, 
  \] 
  which holds for s in the range $-l +1 \leq s \leq m + \yh$. By choosing suitable $l$ and $m$, it holds for any $s \in \R$. 
\end{proof}

Using \rt{Thm2::EstHq}, we can prove the following homotopy formula on $C^\infty$ domains. The proof is identical to that of \rc{Cor::hf}. 

\begin{cor}[Homotopy formula] \label{Cor2::hf} 
	Under the assumptions of \rt{Thm2::EstHq}, let $1\le q\le n$, $s\in\R$ and $1<p<\infty$. Suppose $\varphi\in H^{s,p}_{0,q}(\Om)$ satisfies $\dbar\varphi\in H^{s,p}_{(0,q+1)}(\Om)$. Then the following homotopy formula holds in the sense of distributions: 
	\begin{equation*}
	\varphi=\dbar \Hc_q\varphi+\Hc_{q+1}\dbar\varphi.
	\end{equation*}
	In particular for $\varphi \in H^{s,p}_{(0,q)}(\Om)$ which is $\dbar$-closed, we have $\dbar \Hc_q \var = \var$ and $\Hc_q \var \in H^{s+\frac12,p}_{ (0,q-1)} (\Om)$. 
\end{cor}
In other words we have proved \rt{Thm::Cinfty_bdy_intro}.

\appendix 
\section{Some Results on Sobolev Spaces}\label{SectionA} 
	
\begin{prop}[Equivalent norms] \label{Prop::EqvNorm}
Let $s\in\R$, $1<p<\infty$, and let $\Omega\subset\R^N$ be a bounded Lipschitz domain. Then for $m\in\Z^+$, $H^{s,p}(\Omega)$ have the following equivalent norm:
\begin{equation*}
		\|f\|_{H^{s,p}(\Omega)}\approx\sum_{|\alpha|\le m}\|\partial^\alpha f\|_{H^{s-m,p}(\Omega)}.%,\quad\|f\|_{\Co^s(\Omega)}\approx\sum_{|\alpha|\le m}\|\partial^\alpha f\|_{\Co^{s-m}(\Omega)}. 
		\end{equation*} 
	\end{prop}
	
	This is \cite[Theorem 1.4]{S-Y21-1}.
	
	\begin{prop} \label{H-L}
		Let $1< p < \infty$, $m \in \N$  and $r \in \R$ with $r <m$. Let $\Om$ be a bounded  Lipschitz domain and define $\del(x)$ be the distance function to the boundary $b \Om$. If $u \in W^{m,p}_{\loc} (\Om)$, and $\| \del^{m-r } D^{m} u \|_{L^p(\Om)} < \infty$, then $u \in H^{r,p} (\Om)$. Furthermore, there exists a constant $C$ that does not depend on $u$ such that  
		\[
		\| u \|_{H^{r,p} (\Om)} \leq C \sum_{|\beta| \leq m} \| \del^{m-r} D^{\beta} u \|_{L^p(\Om)}. 
		\]                                         
	\end{prop}	
	\begin{proof}
		For $m =0$, the statement is proved in \cite[Proposition 5.4]{S-Y21-1}, namely for $u \in L^p_{\loc}(\Om)$,  
		\[
		\| u \|_{H^{-s,p}(\Om)} \leq C \| \del^{s} u \|_{L^p (\Om)}, \quad s \geq 0.  
		\] 
		The general case follows by applying  \rp{Prop::EqvNorm}: 
		\begin{equation*}
		\|f\|_{H^{r,p}(\Omega)}\lesssim\sum_{|\alpha|\le m}\|\partial^\alpha f\|_{H^{r-m,p}(\Omega)}\lesssim\sum_{|\alpha|\le m}\|\delta^{m-r}\partial^\alpha f\|_{L^p(\Omega)}, \quad m > r. \qedhere 
		\end{equation*} 
	\end{proof}
	
	\begin{prop}\label{Prop::Sobolev_smoothing}
	    Let $\Om \subset \C^n$ be a bounded Lipschitz domain and let $f \in H^{s,p}_{(0,q)}(\Om)$ with $\dbar f \in H^{s,p}_{(0,q+1)}(\Om)$, then there exists a sequence $f_{\ve} \in C^{\infty}_{(0,q)}(\ov \Om)$ such that 
	    \begin{gather*}
	    f_{\ve} \xrightarrow{\ve \to 0} f \quad \text{in $H^{s,p}(\Om)$;}
	    \\ 
	    \dbar f_{\ve} \xrightarrow{\ve \to 0} \dbar f \quad \text{in $H^{s,p}(\Om)$.}
	    \end{gather*}
	\end{prop}	    
	\begin{proof} 
	First, we show that for a function $g \in H^{s,p}(\C^n)$, 
	there exist $g_{\ve}$ such that 
	 $\| g_{\ve} - g \|_{H^{s,p}(\C^n)} \to 0$. Let $\psi \in C^{\infty}_c(-K)$, where $K$ is some cone centered at $0$; we set $\psi_{\ve}:= \ve^{-N} \psi(\ve^{-1}x) $ and $g_{\ve} := \psi_{\ve} \ast g$. We have
  \begin{align*}
   \| g_{\ve} - g \|_{H^{s,p}(\C^n)}
&= \| \psi_{\ve} \ast g  - g \|_{H^{s,p}(\C^n)}
\\ &= \| (I- \Del)^{\frac{s}{2}} (\psi_{\ve} \ast g) - (I- \Del)^{\frac{s}{2}} g  \|_{L^p(\C^n)} 
  \end{align*}
We claim that $(I - \Del)^{\frac{s}{2}}(\psi_{\ve} \ast g) =
\psi_{\ve} \ast (I - \Del)^{\frac{s}{2}}g $. Indeed, this can be seen by taking Fourier transform: $(1+ 4 \pi |\xi|^2)^{\frac{s}{2}} \hht{\psi_{\ve}} \hht{g} = \hht{\psi_{\ve}} (1+ 4 \pi |\xi|^2)^{\frac{s}{2}} \hht{g}$. Hence 
\[
  \| g_{\ve} - g \|_{H^{s,p}(\C^n)}
  = \| \psi_{\ve} \ast (I- \Del)^{\frac{s}{2}} g - (I- \Del)^{\frac{s}{2}} g  \|_{L^p(\C^n)} \to 0, \quad \text{as $\ve \to 0$}. 
\]
Now let $\Om$ be a bounded Lipschitz domain. 
Take an open covering $\{ U_{\nu} \}_{\nu=0}^M$ of $\Om$ such that 
\[
	U_0  \subset\subset \Om, \quad b \Om \seq \bigcup_{ \nu=1}^M U_{\nu},  \quad U_{\nu} \cap \Om = U_{\nu} \cap \Phi_\nu( \om_{\nu}) ,  \quad \nu=1, \dots, M. 
\]
Here each $\om_{\nu}$ is special Lipschitz domain of the form $\om_{\nu} = \{ x_{2n} > \rho_{\nu} (x') \} $, and $\Phi_\nu$, $1\le \nu\le M$ are invertible affine linear transformations.

Let $(\chi_{\nu})_{\nu=0}^M$ be a partition of unity associated with $\{U_{\nu} \}$, i.e. $\chi_{\nu} \in C^{\infty}_c(U_{\nu}) $ and $\sum_{\nu=0}^M \chi_{\nu} = 1$. We have the property $\om_{\nu} + \Kb = \om_{\nu}$, where $\Kb:= \{ x\in \C^n: x_{2n} >|x'| \}$. We denote $K_\nu=\{\nabla\Phi_\nu\cdot v:v\in \Kb\}$ so that $\Phi_\nu(\omega_\nu)+K_\nu=\Phi_\nu(\omega_\nu)$. 

Take $\psi_0 \in C^{\infty}_{c} (\B^{2n})$ with $\psi_0 \geq 0$ and $\int_{\C^n} \psi_0 = 1$. For $1 \leq \nu \leq {2n}$, take $\psi_\nu\in C^{\infty}_{c} (-K_\nu)$ with $\psi_\nu \geq 0$ and $\int_{\C^{n}} \psi_\nu = 1$. 
Let  $\psi_{\nu, \ve} (x) := \ve^{-{2n}} \psi_\nu (\frac{x}{\ve})$. 
For $\ve>0$ sufficiently small, the convolutions 
$$(\chi_\nu f)\ast \psi_{\nu, \ve}(z)=\int_{-K_\nu} (\chi_{\nu} f) (z - \ve \zeta) \psi_\nu (\zeta) \, d V(\zeta),$$ is defined for $z\in\Omega$.  

Let $\wti{f}$ be an extension of $f$ and set
\[
  \wti{f}_{\ve}:=\sum_{\nu=0}^M (\chi_{\nu} \wti{f} ) \ast \psi_{\nu, \ve}. 
\] 
Clearly $ \wti{f}_{\ve} \in C^{\infty} (\C^n)$. By \cite[Theorem 2.8.2(ii)]{Tri83} (recall from Proposition \ref{Prop::L-P} that $H^{s,p}(\C^n)=\Fs_{p2}^s(\C^n)$), we have $\chi_{\nu} \wti f \in H^{s,p}(\C^n)$ for all $\nu$. 
	
Now by the first part of the proof applied to $g:= \chi_{\nu} \wti f$, $\nu=0,\dots, M$, we obtain  
	$\|\wti{f} _{\ve} - \sum_{\nu=0}^{M} \chi_{\nu} \wti f \|_{H^{s,p}(\C^n)}
	= \|\wti{f} _{\ve} - \wti{f} \|_{H^{s,p}(\C^n)} \to 0$ as $\ve \to 0$. 
By taking restriction $f_{\ve}:= \wti{f}_{\ve}|_{\Om}$, we obtain $\| f_{\ve} - f \|_{H^{s,p}(\Om)} \to 0$ as $\ve \to 0$. 

Finally, to show that $\| \dbar f_{\ve} - \dbar f \|_{H^{s,p}(\Om)} \to 0$, it suffices to show that $\| \dbar\wti f_{\ve} - \dbar\wti f \|_{H^{s,p}(\C^n)} \to 0$ since $\dbar \wti{f_{\ve}}|_{\Om} = \dbar f_{\ve}$ and $\dbar \wti f|_{\Om} = \dbar f$. We have 
\begin{equation} \label{der_cvg}
   \dbar \wti{f}_{\ve} = 
  \sum_{\nu=0}^M (\chi_{\nu} \dbar \wti{f}) \ast \psi_{\nu, \ve}
  + \sum_{\nu=0}^M (\dbar\chi_{\nu}\wedge \wti{f}) \ast \psi_{\nu, \ve} 
\end{equation}
By the first part of the proof applied to $g = \chi_{\nu} \dbar\wti f$ and $\dbar\chi_{\nu}\wedge f$, the first sum in \re{der_cvg} converges to 
$\dbar \wti{f}$, while the sum of last two terms converges to $\sum_{\nu= 0}^{\infty} (\dbar \chi_{\nu})\wedge \wti f = 0$ in $H^{s,p}(\C^n)$. This shows the claim $\|\dbar f_\eps-\dbar f\|_{H^{s,p}(\Omega)}\to0$ and concludes the whole proof. 
\end{proof}

	\bibliographystyle{amsalpha}
	\bibliography{master_CK} 
	
\end{document}